%% file: main.tex
\documentclass[preprint,11pt,authoryear]{elsarticle}

\usepackage{amsmath,amsfonts}
\usepackage{amsthm}
\usepackage{bm}

\newtheorem{proposition}{Proposition}

\newtheorem{remark}{Remark}

\usepackage{siunitx}
\usepackage{booktabs}
\usepackage{makecell}
\usepackage{multirow}

\usepackage{enumitem}
\usepackage{mathtools}
\usepackage{algorithmic}
\usepackage{algorithm}
\usepackage{array}
\usepackage[caption=false,font=footnotesize]{subfig}
\usepackage{textcomp}

\usepackage{stfloats}
\usepackage{url}
\usepackage{verbatim}
\usepackage{graphicx}
\DeclareMathOperator*{\argmin}{arg\,min}

\graphicspath{{Figure/}}
\usepackage[a4paper,includehead,includefoot,
            top=20mm,bottom=24mm,left=18mm,right=18mm,
            headsep=8mm,footskip=12mm]{geometry}

\usepackage[dvipsnames]{xcolor}
\usepackage[colorlinks,linkcolor=blue,citecolor=ForestGreen]{hyperref}

\usepackage{soul}

\newcommand{\DA}{\mathrm{DA}}
\newcommand{\RT}{\mathrm{RT}}

\newcommand{\ES}{\mathrm{ES}}
\newcommand{\PV}{\mathrm{PV}}

\usepackage{accents}
\newcommand{\ubar}[1]{\underaccent{\bar}{#1}}
\newcommand{\utilde}[1]{\underaccent{\tilde}{#1}}

\usepackage[normalem]{ulem}

\allowdisplaybreaks % break align environment

\usepackage{amssymb} % for varnothing

\journal{European Journal of Operational Research}

\begin{document}

\begin{frontmatter}



\title{Day-Ahead Offering for Virtual Power Plants: A Stochastic Linear Programming Reformulation and Projected Subgradient Method} %% Article title


% \author{Weiqi Meng, Hongyi Li, and Bai Cui} %% Author name

% \affiliation{organization={Department of Electrical
% and Computer Engineering, Iowa State University},%Department and Organization
%             % addressline={2215 Coover Hall 2520 Osborn Drive}, 
%             city={Ames},
%             postcode={50010}, 
%             state={Iowa},
%             country={USA}}

\author{Weiqi Meng}
\ead{mengwq@iastate.edu}

\author{Hongyi Li}
\ead{hongyili@iastate.edu}

\author{Bai Cui\corref{cor1}}
\ead{baicui@iastate.edu}

\cortext[cor1]{Corresponding author}

\affiliation{organization={Department of Electrical and Computer Engineering, Iowa State University},
            city={Ames},
            postcode={50010},
            state={Iowa},
            country={USA}}

%% Abstract
\begin{abstract}
Virtual power plants (VPPs) are an emerging paradigm that aggregates distributed energy resources (DERs) for coordinated participation in power systems, including bidding as a single dispatchable entity in the wholesale market. In this paper, we address a critical operational challenge for VPPs: the day-ahead offering problem under highly intermittent and uncertain DER outputs and market prices. The day-ahead offering problem determines the price-quantity pairs submitted by VPPs while balancing profit opportunities against operational uncertainties.
First, we formulate the problem as a scenario-based two-stage stochastic adaptive robust optimization problem, where the uncertainty of the locational marginal prices follows a Markov process and DER uncertainty is characterized by static uncertainty sets. Then, motivated by the outer approximation principle of the column-and-constraint generation (CC\&G) algorithm, we propose a novel inner approximation-based projected subgradient method. By exploiting the problem structure, we propose two novel approaches to improve computational tractability. First, we show that under mild modeling assumptions, the robust second-stage problem can be equivalently reformulated as a linear program (LP) with a nested resource allocation structure that is amenable to an efficient greedy algorithm. Furthermore, motivated by the computational efficiency of solving the reformulated primal second-stage problem and the isotonic structure of the first-stage feasible region, we propose an efficient projected subgradient algorithm to solve the overall stochastic LP problem. Extensive computational experiments using real-world data demonstrate that the overall projected subgradient descent method achieves about two orders of magnitude speedup over CC\&G while maintaining solution quality. 
\end{abstract}


% \centering
% \includegraphics[width=4.2 in]{Figure1.1.pdf}


% \begin{highlights}
% \item A two-stage stochastic LP day-ahead offering counterpart model is presented
% \item Tractable second-stage problem is derived, and solved by a greedy algorithm
% \item Offering problem is solved by inner approximation-based projected subgradient algorithm
% \item Our method achieves two orders of magnitude total speedup over state-of-the-art algorithm
% % , via reformulation $(2\times)$ and algorithmic design $(70\times)$

% \item Real-world data is used for numerical studies with various benchmarks

% \end{highlights}

%% Keywords

\end{frontmatter}



\section{Introduction}
\subsection{Background and Motivation}
\label{Background and Motivation}
%% Labels are used to cross-reference an item using \ref command.

Distributed energy resources (DERs) are reshaping electric power systems by enabling localized rooftop photovoltaics (PVs), behind-the-meter energy storage units, and demand flexibility at the distribution level \cite{iea_unlocking_2022}. Virtual power plants (VPPs) aggregate DERs through coordination platforms to provide grid services comparable to conventional generators \cite{kim_operations_2022}. The aggregation transforms dispersed residential and commercial assets into wholesale market participants \cite{us_department_of_energy_pathways_2025}.

In U.S. electricity markets, utility-based DER aggregators submit stepwise offering curves ahead of the operation day \cite{lin_electricity_2017}. The Independent System Operator (ISO) clears the market using aggregated bids and offers, determining the locational marginal prices (LMPs) and accepted quantities for each participant \cite{plazas_multimarket_2005}. In vertically integrated or non-ISO regions, analogous mechanisms are used to determine schedules and settlement prices. Once settled, aggregators need to commit the offer as the uncertainty of all DERs is realized during the operation day. Hence, VPPs should develop strategic offers to maximize their profits. This day-ahead market structure requires a two-stage adaptive decision-making framework, where the day-ahead offering decisions are made here-and-now, and the second-stage operational decisions are made in real-time to respond dynamically to updated uncertainty information \cite{sun_robust_2021}. 
% This two-stage adaptive decision-making model is challenging primarily due to the curse of dimensionality of the price scenario, and 
Existing research focuses on the trade-off between computational time and solution quality and often fails to reach the optimal solution in a reasonable time \cite{fischetti_deep_2018}. Two fundamental computational bottlenecks exist for this problem: scenario explosion in day-ahead price modeling and the VPP dispatch decision-making under uncertainties in real-time operation. The first stems from the exponential growth of scenario trajectories with increasing temporal resolution, while the second reflects the need to solve a nontrivial nonconvex second-stage problem that implements the first-stage commitments under the realized price trajectory. 
In this work, we develop an efficient algorithmic framework for solving the two-stage stochastic offering problem that addresses the computational bottlenecks above.
% As corroborated by extensive computational experiments, our tailored algorithmic framework is able to achieve orders-of-magnitude speedup over conventional approaches while maintaining solution quality.

\subsection{Literature Review}
\label{Literature Review}

The OR\&MS literature has extensively examined strategic offering/bidding behavior within day-ahead electricity markets \cite{rintamaki_strategic_2020} \cite{sunar_strategic_2019}. In existing models, DER aggregators (or VPPs) submit optimal offer curves to earn profit from the wholesale market; subsequently, the second-stage problem addresses the DER real-time dispatch problem \cite{rintamaki_strategic_2020}. The accurate formulation of two-stage adaptive decision-making relies on the modeling of heterogeneous uncertainties, such as price and DER production \cite{yu_uncertainties_2019}.

Early research employed deterministic optimization frameworks for the offering problem \cite{hellmers_operational_2016}. These approaches proved inadequate with high penetration levels of DERs, prompting adoption of stochastic programming (SP) \cite{kraft_stochastic_2023} and robust optimization (RO) \cite{sun_robust_2021} methodologies. As for SP, reference \cite{plazas_multimarket_2005} first propose an ARIMA-based price scenario generation method. However, deriving credible probabilistic price distributions remains methodologically challenging \cite{fan_min-max_2014}. To address “curse of dimensionality”, scenario reduction techniques are proposed accordingly  \cite{plazas_multimarket_2005}. Reference \cite{kim_sample_2024} proposes a Markov decision process-based approximate dynamic programming for multi-stage bidding strategies, while \cite{zheng_arbitraging_2022} quantifies the energy storage unit optionality across real-time markets using similar techniques.
Following \cite{kim_sample_2024} and \cite{zheng_arbitraging_2022}, we adopt a Markov process to model the day-ahead price and generate finite sample price trajectories accordingly. Note that this price modeling approach still faces a dimensionality challenge: the total number of price trajectories is exponential in the number of discretized price states per hour. Wholesale market protocols also require careful definition of the real-time LMPs, which are linked to the day-ahead price \cite{kim_benefits_2021, krishnamurthy_energy_2018}. To reflect the anti-arbitrage logic embedded in U.S. two-settlement markets, we represent real-time imbalance settlement through an ``incentive-penalty'' mechanism linked to the day-ahead market prices. This stylized term follows the practical requirement that intentional deviations from day-ahead commitments should not be rewarded. This term captures market fundamentals, but it requires specialized decomposition techniques to maintain computational tractability. 

RO presents an alternative paradigm for VPP offering strategies. Its distribution-free nature circumvents the need for precise probabilistic characterizations, requiring only uncertainty bounds that prove simpler to estimate empirically \cite{fan_min-max_2014}. This framework guarantees feasibility across all uncertainty realizations within prescribed sets. Adaptive robust optimization (ARO) extends this approach by enabling tunable conservativeness through budget parameters \cite{baringo_stochastic_2018}. Reference \cite{baringo_offering_2011} pioneered RO application to price-taker bidding curves, demonstrating significant computational advantages over SP counterparts \cite{ben-tal_robust_2009}. Multi-stage formulations have subsequently emerged through ARO frameworks \cite{attarha_affinely_2019}. Despite computational benefits, the inherent limitation of RO lies in its conservatism. This limitation motivates the integration of SP and RO, leveraging their complementary strengths \cite{baringo_stochastic_2018}.

Two-stage adaptive robust optimization (2S-ARO) frameworks dominate current VPP offering models, as comprehensively surveyed in \cite{sinha_review_2018}, \cite{kleinert_survey_2021} and \cite{beck_survey_2023}. For mixed-asset VPPs integrating diverse generation resources, computational efficiency becomes paramount for practical deployment. The fundamental challenge lies in the strong NP-hardness of 2S-ARO \cite{sun_robust_2021}---even finding an $\epsilon$-approximate solution lacks polynomial-time guarantees. Reference \cite{bertsimas_adaptive_2013} addresses this intractability through an alternating direction method that accepts suboptimal second-stage problem solutions. Subsequently, reference \cite{zeng_solving_2013} introduces the column-and-constraint generation (CC\&G) algorithm, substantially accelerating convergence compared to traditional Benders decomposition. However, the CC\&G algorithm faces scalability limitations when handling large-scale scenario trajectories and polyhedral uncertainty sets. Two fundamental challenges arise: First, the computational burden of solving the master problem (MP) escalates significantly as the number of scenarios grows, since each iteration introduces additional copies of second-stage variables. Commercial solvers like Gurobi struggle to achieve optimality even for moderate-sized MPs after several iterations. Second, solving the second-stage problem requires nonconvex max-min optimization when constraint uncertainty is present, compounding computational complexity. Three solution strategies are commonly used to solve the robust second-stage problem \cite{kleinert_survey_2021}: (i) KKT reformulation, where complementarity conditions are linearized via big-M methods, yielding mixed-integer linear programs (MILPs) \cite{kang_stochastic-robust_2023}; (ii) dual transformation, converting nonconvex max-min structures to single-level formulations through appropriate bilinear term linearization \cite{bertsimas_adaptive_2013}; and (iii) heuristic approaches that trade optimality for computational efficiency \cite{shao_neural_2025}.

\subsection{Contribution}

In this study, we propose a two-stage stochastic ARO formulation, reformulate it as a two-stage stochastic linear programming (LP) problem, and develop corresponding solution algorithms for the optimal VPP day-ahead offering problem. The existing literature suggests that the computational tractability of the problem is fundamentally limited by the curse of dimensionality of SP and the computational issues of solving large-scale 2S-ARO problems. To circumvent these limitations, the following four contributions are made: 
% Three key observations emerge from existing literature: (i) two-stage stochastic ARO necessities uncertainty modeling methods adhering to market rules; (ii) current SP methods rely on scenario reduction, limiting scalability; and (iii) there are two aspects of the CC\&G algorithm that make it difficult to implement in practice: The MP becomes challenging to solve as the number of scenarios increases due to multiple copies of second-stage variables; the second-stage problem requires solving a nonconvex max-min optimization for problems with constraint uncertainty. This work presents a novel two-stage stochastic ARO offering model for VPPs. The contributions of this study are fourfold:
\begin{enumerate}

\item Second-stage problem reformulation: In Section~\ref{reformulated formulation}, we exploit the cumulative-sum storage structure to reformulate the second-stage problem as a structured LP with piecewise-linear cost. This structural reformulation enables a fast value-function oracle for repeated scenario-wise recourse evaluations and subgradient computation in the outer algorithm.

\item Max-min decoupling: We show in Section~\ref{reformulated formulation} that under mild modeling assumptions, the nonconvex robust second-stage DER dispatch problem can be equivalently reformulated as an LP, which simplifies the overall problem as a two-stage stochastic LP problem.

% \item Problem reformulation in Section~\ref{reformulated formulation}: Power balance constraints involving first-stage decisions are embedded into the objective of recourse problem to enable gradient-based optimization. The cumulative sum constraint structure, combined with the reformulated PWL objective, enables efficient greedy algorithms.

\item Greedy algorithm for solving the second-stage problem: An efficient custom greedy algorithm is proposed in Section~\ref{Greedy Algorithm} for the reformulated second-stage problem. This algorithm synergizes with the overall algorithmic framework and yields a clear computational advantage over a general-purpose LP solver in repeated scenario-wise evaluation.

\item Projected subgradient method: A projected subgradient method with pool-adjacent-violators algorithm (PSM-PAVA) is designed in Section~\ref{Projected Subgradient Method} to solve the overall problem efficiently. This tailored gradient-informed inner approximation algorithm outperforms benchmark algorithms with significant total speedup.

% the cut-based outer approximation in CC\&G/Benders decomposition, achieves significant speed-up.

\end{enumerate}

\subsection{Paper Organization}
This paper proceeds as follows. Section~\ref{sect:formulation} establishes the mathematical formulation of the two-stage stochastic ARO model. The reformulation is derived in Section~\ref{reformulated formulation}. Section~\ref{algorithm design part} develops algorithms, including a greedy algorithm and PSM-PAVA. Extensive case studies are presented in Section~\ref{case study}. Finally, Section~\ref{conclusion} concludes the paper.

\input{section_formulation_revised_MWQ.tex}

\input{reformulation_cost-to-go.tex}

\section{Algorithm Design}
\label{algorithm design part}

The overall two-stage stochastic LP problem is highly structured, which can be exploited to develop efficient algorithms. In this section, a computationally efficient oracle is designed to solve the decoupled second-stage LP problem in \eqref{eq:2S-LP}. With this oracle, we then introduce the PSM algorithm to effectively solve the two-stage stochastic LP problem.

\input{LSA.tex}

\input{PSM.tex}

\subsection{Overall Algorithm}

The complete PSM-PAVA procedure is summarized in Algorithm~\ref{alg:psg-pava}. At each iteration, the greedy oracle evaluates the second-stage problem for all scenarios and identifies the active slopes, from which the subgradient is computed by \eqref{eq:subgrad-full}. The iterate is then updated by a projected subgradient step, where feasibility is enforced by the PAVA-based projection in \eqref{eq:isotonic-proj}. The process continues until the practical termination criterion in \eqref{eq:stopping} is satisfied.

\begin{algorithm}[t]
\caption{PSM-PAVA for the two-stage stochastic LP problem}
\label{alg:psg-pava}
\begin{algorithmic}[1]
\REQUIRE Initial offering $p^{\mathrm{DA},(0)} \in \mathcal{X}$, initial step size $\alpha_0$, scenario set $\{\omega,\rho_\omega\}_{\omega\in\Omega}$, tolerance $\varepsilon_{\mathrm{rel}}$, step size bounds $\alpha_{\min},\alpha_{\max}$, maximum iterations $K_{\max}$
\ENSURE Incumbent solution $p^{\mathrm{DA},*}$ and objective value $\Phi(p^{\mathrm{DA},*})$

\STATE Initialize $\Phi^{\mathrm{best}} \gets +\infty$, $p^{\mathrm{DA},\mathrm{best}} \gets p^{\mathrm{DA},(0)}$

\FOR{$k=0,1,\dots,K_{\max}-1$}
    \FORALL{$\omega\in\Omega$}
        \STATE Compute worst-case PV profile $\hat{p}^{\mathrm{PV},*}(\omega)$ via \eqref{eq:worst_case_pv_prop}
        \STATE Solve the second-stage LP problem in \eqref{eq:2S-LP} by Algorithm~\ref{alg:local_search_algorithm}; obtain $V_\omega^{(k)}$ and active slopes
    \ENDFOR
    \STATE Compute $\Phi(p^{\mathrm{DA},(k)})$ and $g^{(k)}$ via \eqref{eq:subgrad-full}    \IF{$\Phi(p^{\mathrm{DA},(k)}) < \Phi^{\mathrm{best}}$}
        \STATE $p^{\mathrm{DA},\mathrm{best}} \gets p^{\mathrm{DA},(k)}$, $\Phi^{\mathrm{best}} \gets \Phi(p^{\mathrm{DA},(k)})$
    \ENDIF
    \IF{$k=0$}
        \STATE Set $\alpha_k \gets \alpha_0$
    \ELSE
        \STATE Compute $\alpha_k$ by \eqref{eq:spectral-step} if applicable; otherwise set $\alpha_k \gets \alpha_0$; clip $\alpha_k$ to $[\alpha_{\min},\alpha_{\max}]$
    \ENDIF
    \STATE Compute $p^{\mathrm{DA},(k+1)} \gets \Pi_{\mathcal{X}}(p^{\mathrm{DA},(k)}-\alpha_k g^{(k)})$ via PAVA; evaluate $\Phi(p^{\mathrm{DA},(k+1)})$
    \IF{$\Phi(p^{\mathrm{DA},(k+1)}) < \Phi^{\mathrm{best}}$}
        \STATE $p^{\mathrm{DA},\mathrm{best}} \gets p^{\mathrm{DA},(k+1)}$, $\Phi^{\mathrm{best}} \gets \Phi(p^{\mathrm{DA},(k+1)})$
    \ENDIF
    \IF{\eqref{eq:stopping} is satisfied}
        \STATE \textbf{break}
    \ENDIF
\ENDFOR
\STATE \textbf{return} $p^{\mathrm{DA},*} \gets p^{\mathrm{DA},\mathrm{best}}$, \ $\Phi(p^{\mathrm{DA},*}) \gets \Phi^{\mathrm{best}}$
\end{algorithmic}
\end{algorithm}

\begin{figure}[!t]
\centering
\includegraphics[width=5 in]{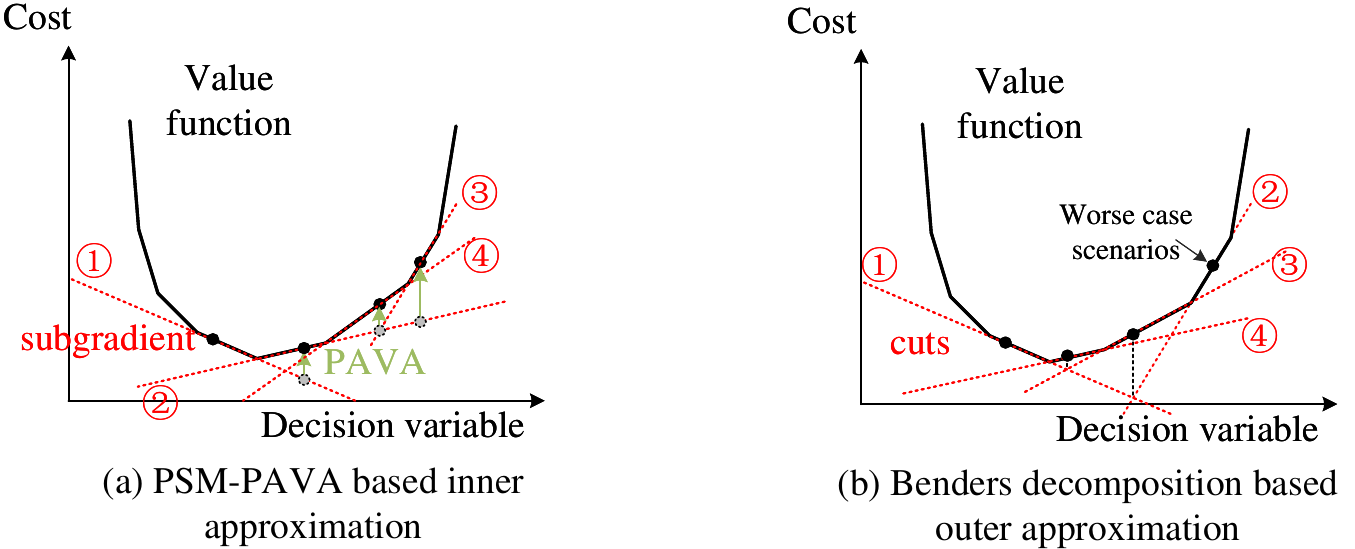} % 请将fig1替换为实际的图像文件名
\caption{Two ways of approximation of the value function.}
\label{InnerApproximation}
\end{figure}

The key distinction between the PSM-PAVA algorithm and Benders decomposition \cite{sun_robust_2021} lies in the approximation of the value function $V_\omega(p^\DA)$. As illustrated in Figure~\ref{InnerApproximation}, Benders decomposition builds an \emph{outer approximation} of $V_\omega(p^\DA)$ by iteratively adding cuts to the master problem. Although these cuts (value and the subgradient of the second-stage problem) progressively tighten the approximation, they also enlarge the master problem and increase the per-iteration computational cost. In contrast, the PSM-PAVA algorithm can be interpreted as a gradient-informed first-order local \emph{inner approximation} strategy. At each iterate $p^{\mathrm{DA},(k)}$, the greedy oracle evaluates the scenario-wise $V_\omega(p^{\mathrm{DA},(k)})$ exactly and returns a valid subgradient, which is then used to generate a feasible projected update. Hence, unlike Benders decomposition, PSM-PAVA does not accumulate cuts or expand a master problem, but instead does not increase the size of the optimization model across iterations. This is one key reason for the computational advantage reported in Section~\ref{case study}. It should be noted that the comparison reflects the total computational benefit of the combined reformulation-and-algorithm framework, which is the relevant metric for practical deployment.

\section{Case Study}
\label{case study}

This section evaluates the proposed model and algorithm through numerical case studies examining computational efficiency, solution quality, and scalability. 

\subsection{Experiment Setup}
Numerical case studies utilize real operational data from multiple sources in \cite{us_department_of_energy_pathways_2025}. Historical day-ahead price data are obtained from U.S. utilities. DER parameters and operating profiles are compiled from distribution-level utility measurements, including: (i) 10-kW rooftop solar PV; (ii) battery energy storage units; and (iii) deterministic aggregated load. For clarity, the main DER configurations and corresponding data sources used for the VPP in the case study are summarized in Table~\ref{tab:DER_setup}. To reflect realistic operational conditions, we construct a utility-operated VPP comprising 200 PV units and 100 energy storage units. At each time slot, five electricity price states are specified, which define a Markov-chain scenario space with $5^{24}$ possible day-ahead price trajectories. In the experiments, a finite number of representative day-ahead price trajectories are sampled from this space, as illustrated in Figure~\ref{fig:markov}. We note that, the sampled day-ahead prices satisfy the assumption in Remark \ref{remark real time price}: $\lambda_{t,\omega}^{\DA} > c^{\PV}  > 0,  \forall t\in\mathcal{T},\ \forall\omega\in\Omega$. The budgeted uncertainty set for solar PV generation over the daily horizon is modeled in previous work \cite{Meng2025DER}, as shown in Figure~\ref{fig:PV}. The experiments are executed in MATLAB (R2024a) on a 64-bit Windows 11 workstation with an Intel\textsuperscript{\textregistered} Core\texttrademark{} i7-6700HQ CPU running at 2.64 GHz and 24 GB of RAM, using YALMIP for modeling and GUROBI 12.0 as the solver.

% \begin{figure}[!t]
% \centering
% \includegraphics[width=5.2 in]{Figure/Parameters.png} % 请将fig1替换为实际的图像文件名
% \caption{Pre-determinted uncertainty parameters: (1) Sampled 4 day-ahead price scenarios through Markov process; (2) Static uncertainty set of PV generation.}
% \label{MarkovandPV}
% \end{figure}

\begin{figure}
    \centering
    \subfloat[ Sampled four price scenarios]{%
        \includegraphics[width=0.3\linewidth]{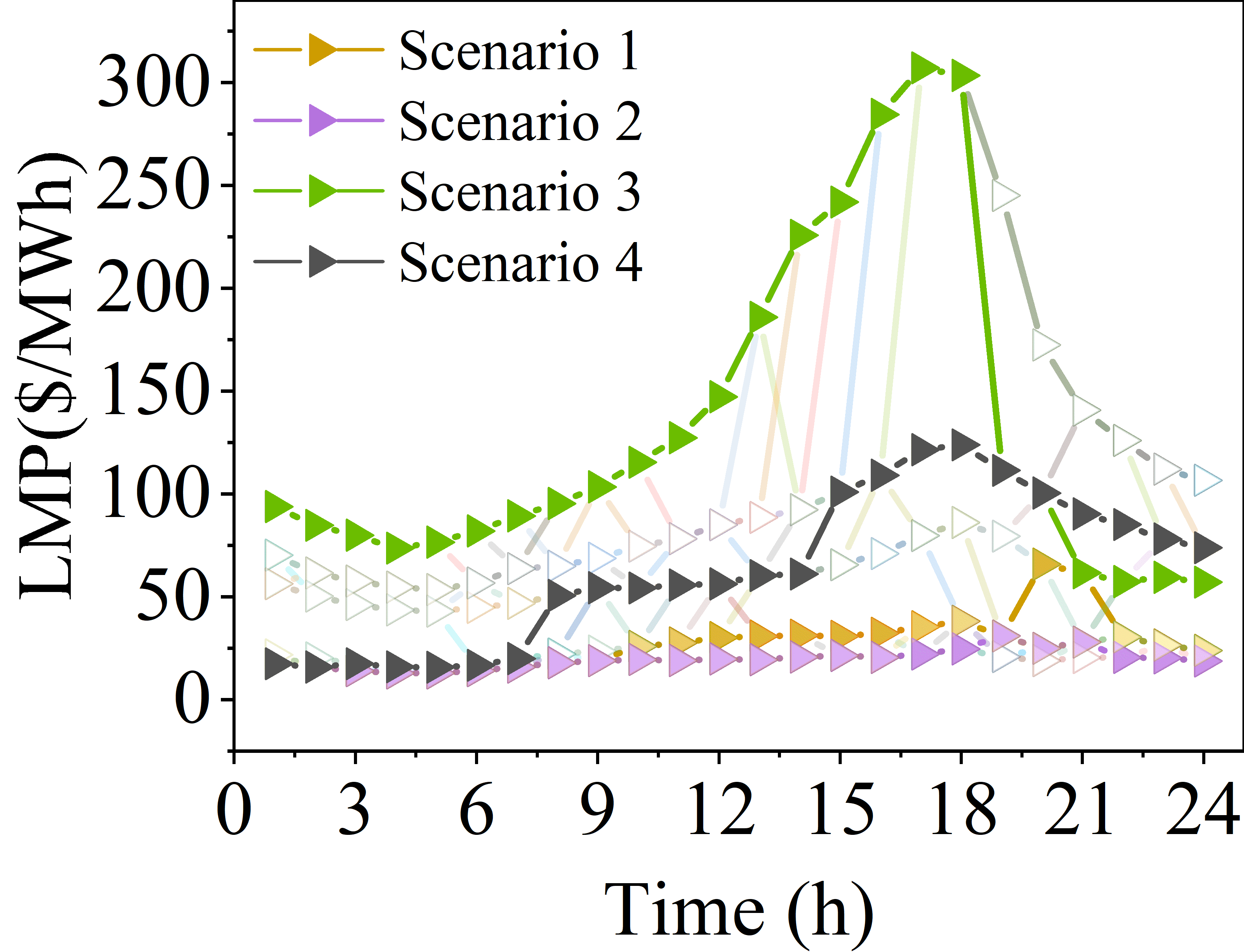}
        \label{fig:markov}}
    \qquad
    \subfloat[Static PV uncertainty set]{%
        \includegraphics[width=0.3\linewidth]{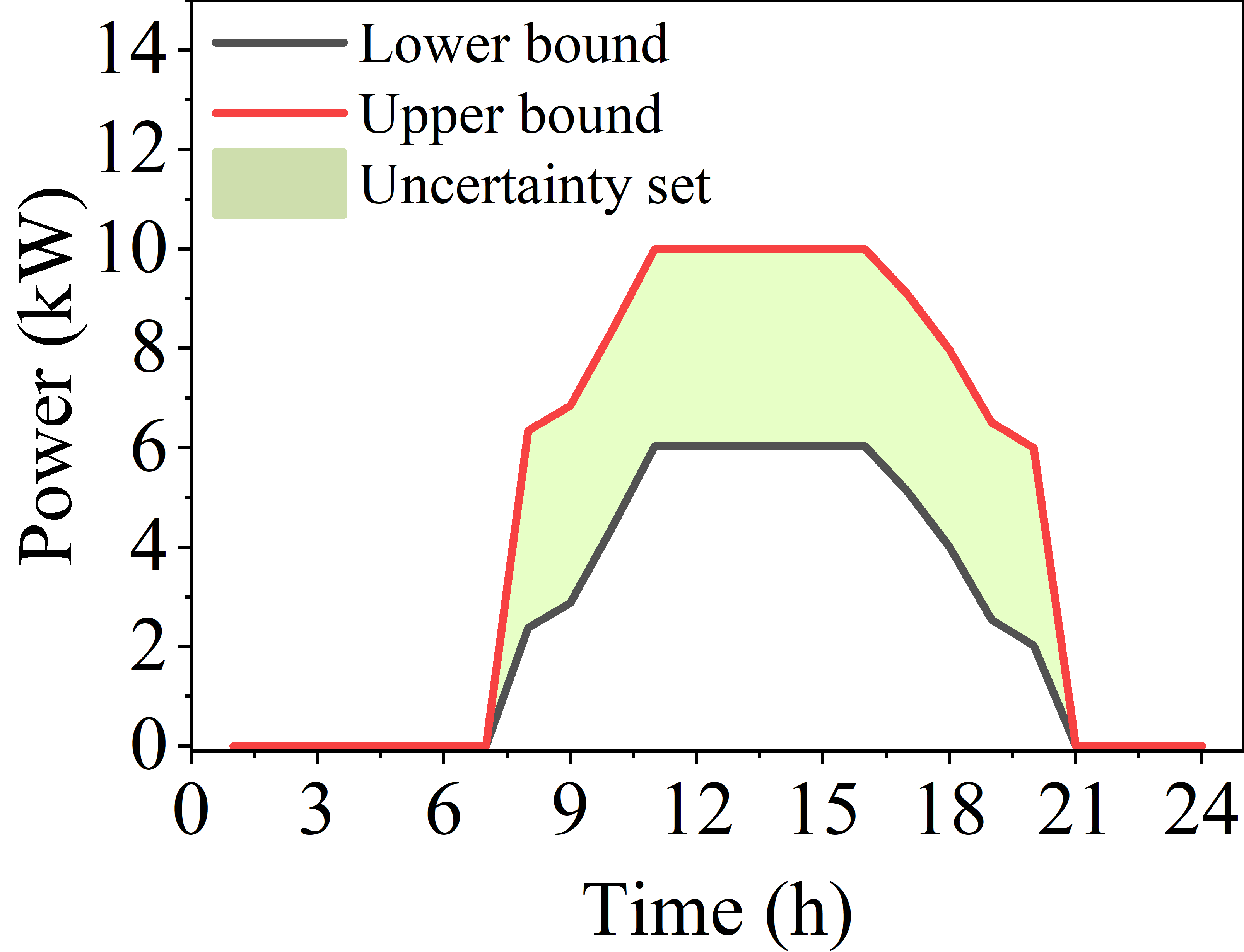}
        \label{fig:PV}}
    \\
    \caption{Pre-determined uncertainty parameters.}
    \label{MarkovandPV}
\end{figure}

\begin{table}[t]
  \centering
  \caption{DER configurations used in the case study.}
  \label{tab:DER_setup}
  \scriptsize
  \setlength{\tabcolsep}{3pt}
  \renewcommand{\arraystretch}{1.05}
  \sisetup{detect-weight=true,detect-inline-weight=math,table-number-alignment=center}

  \begin{tabular}{
    l
    S[table-format=3.0]
    S[table-format=2.1]
    S[table-format=2.1]
    l
  }
  \toprule
  \multicolumn{1}{c}{Component} &
  \multicolumn{1}{c}{Number of units} &
  \multicolumn{1}{c}{Power rating (kW)} &
  \multicolumn{1}{c}{Energy rating (kWh)} &
  \multicolumn{1}{c}{Data source} \\
  \midrule
  Rooftop solar PV         & 200 & 10.0 & {--}  & Utility historical generation records \\
  Battery storage unit   & 100 & 11.3 & 14.5  & Distribution-level utility measurements \\
  Aggregated load        & {--} & {--} & {--} & Utility smart-meter data \\
  \bottomrule
  \end{tabular}
\end{table}

\subsection{Comparison of Greedy Algorithm and LP Solver}
\label{Comparison of Greedy Algorithm}

The first case study evaluates the computational performance of the greedy algorithm. To assess the performance of the greedy algorithm, we benchmark it against the advanced commercial solver Gurobi 12.0. Both methods employ the same outer-level PSM-PAVA algorithm; the key distinction lies in their respective solution strategies for the second-stage problem. For precision in time measurements, in this paper, when the number of scenarios is small and the total computation time falls below one second, the algorithm is executed repeatedly across multiple sampling runs to determine the average solution time per instance. 

To establish the generalizability of our findings, we examine three distinct problem scales with $|\Omega|=$25, 100, and 250 price scenarios, respectively. All scenarios are generated using the Markov process presented in Section~\ref{sect:Markov_Process}. Figure~\ref{fig:lsa_vs_solver_all} illustrates the convergence profiles of the optimal objective values of the two-stage stochastic LP problems across different scenario settings. The computational time per iteration is shown in a box plot. A comprehensive performance analysis comparing the greedy algorithm and the benchmark LP solver for the two-stage stochastic LP problem is provided in Table~\ref{tab:greedy algorithm_vs_solver}.

\begin{table}[t]
  \centering
  \caption{Performance comparison for greedy algorithm and LP solver (Number of iteration=50, $\Gamma=6$).}
  \label{tab:greedy algorithm_vs_solver}
  % ---- single-column friendly tweaks ----
  \scriptsize
  \setlength{\tabcolsep}{2.5pt}
  \renewcommand{\arraystretch}{1.05}
  \sisetup{detect-weight=true,detect-inline-weight=math,table-number-alignment=center}

  \begin{tabular}{
    S[table-format=3.0]
    S[table-format=2.2] S[table-format=2.2]
    S[table-format=2.2] S[table-format=2.2]
    S[table-format=4.3] S[table-format=4.5]
    S[table-format=1.3,table-space-text-post = \%]
  }
  
\toprule
\multicolumn{1}{c}{\multirow{2}{*}{Number of scen.}} & \multicolumn{2}{c}{Second-stage LP (s)} &
\multicolumn{2}{c}{Second-stage LP (ms/iter)} &
\multicolumn{2}{c}{Two-stage stochastic LP obj (\$)} &
\multicolumn{1}{c}{\multirow{2}{*}{Diff (\%)}} \\

\cmidrule(lr){2-3}\cmidrule(lr){4-5}\cmidrule(lr){6-7}
\multicolumn{1}{c}{} & \multicolumn{1}{c}{Greedy algorithm} & \multicolumn{1}{c}{LP} &
\multicolumn{1}{c}{Greedy algorithm} & \multicolumn{1}{c}{LP} &
\multicolumn{1}{c}{Greedy algorithm} & \multicolumn{1}{c}{LP} & \multicolumn{1}{c}{} \\
  \midrule
   25  & \multicolumn{1}{c}{\makecell[c]{1.16\\(3.29$\times$)}} & 3.82 &  \multicolumn{1}{c}{\makecell[c]{4.58\\(2.21$\times$)}} & 10.13 & 1891.35 & 1891.34 & 0.00053 \\
  100  & \multicolumn{1}{c}{\makecell[c]{4.01\\(1.99$\times$)}} & 7.96 & \multicolumn{1}{c}{\makecell[c]{19.66\\(1.98$\times$)}} & 38.98 & 1668.70 & 1668.68 & 0.0012 \\
  250  & \multicolumn{1}{c}{\makecell[c]{9.42\\(1.70$\times$)}} & 15.97 & \multicolumn{1}{c}{\makecell[c]{47.42\\(2.06$\times$)}} & 97.65 & 1618.21 & 1618.21 & 0.00001 \\
  \bottomrule
  \end{tabular}
\end{table}

% \begin{figure}[!t]
% \centering
% \includegraphics[width=6.2 in]{Figure/LSAvsSolver.png} % 请将fig1替换为实际的图像文件名
% \caption{Performance comparison of the iteration process, optimality and computational time for the proposed greedy algorithm and the benchmark commercial LP solver for the second-stage nonconvex max-min problem of the 2S-ARO problem (Maximum iteration =50; $\Gamma=6$).}
% \label{greedy algorithmvsSolver}
% \end{figure}

\begin{figure*}[!t]
\centering

% ------------------ Row 1 ------------------
\subfloat[Iter./Opt. (|$\Omega$|=25)]{
    \includegraphics[width=0.30\textwidth]{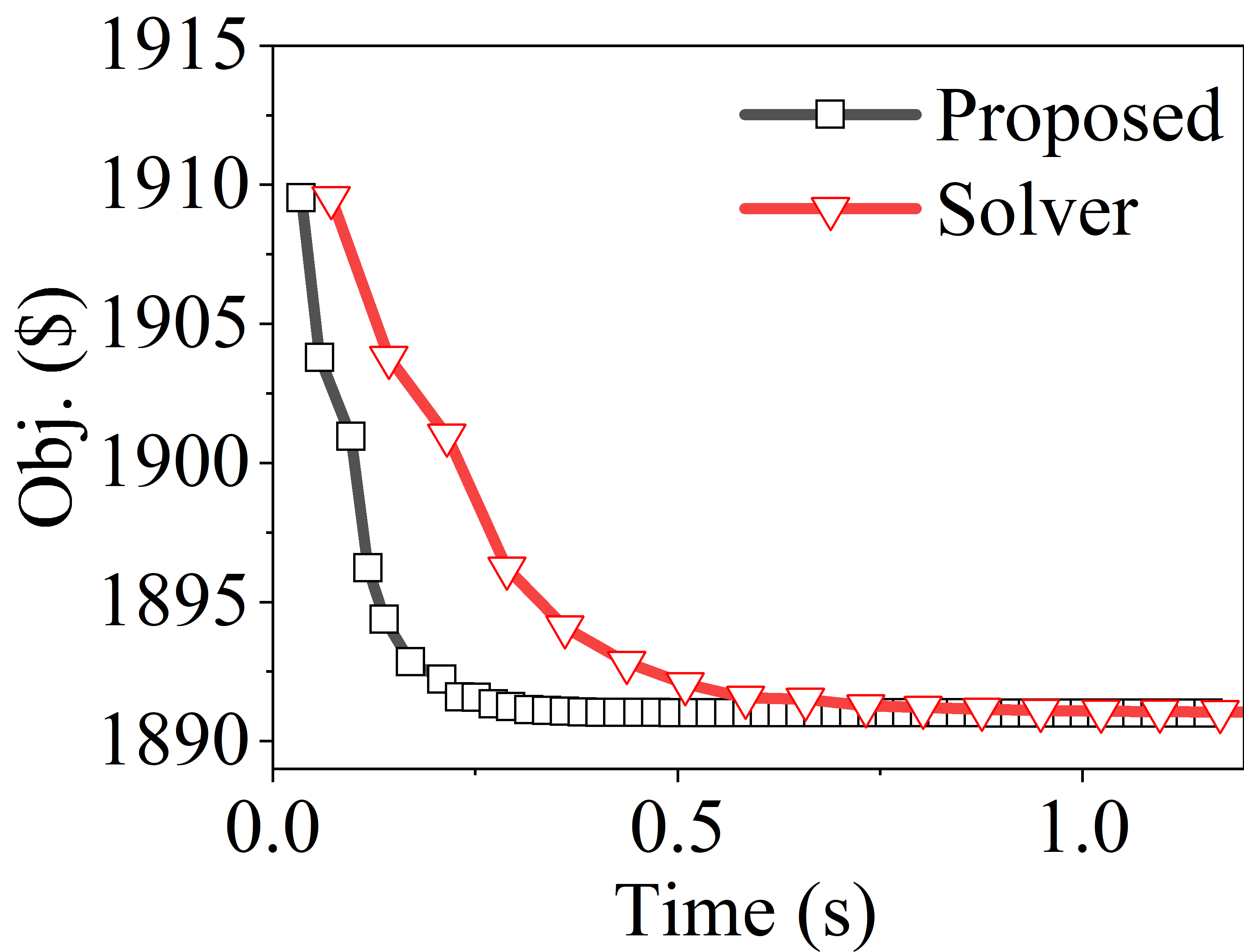}
}
\hfill
\subfloat[Iter./Opt. (|$\Omega$|=100)]{
    \includegraphics[width=0.30\textwidth]{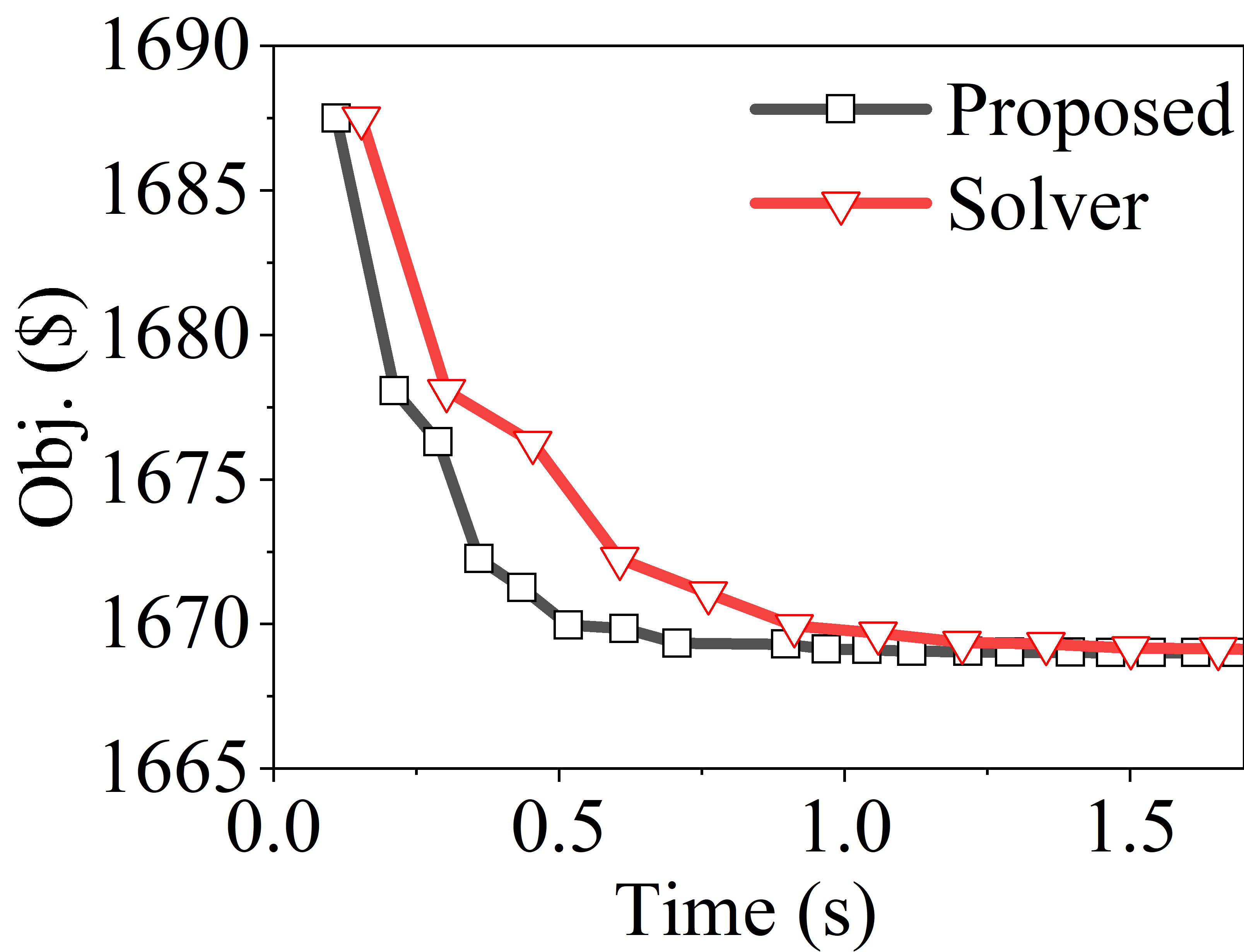}
}
\hfill
\subfloat[Iter./Opt. (|$\Omega$|=250)]{
    \includegraphics[width=0.30\textwidth]{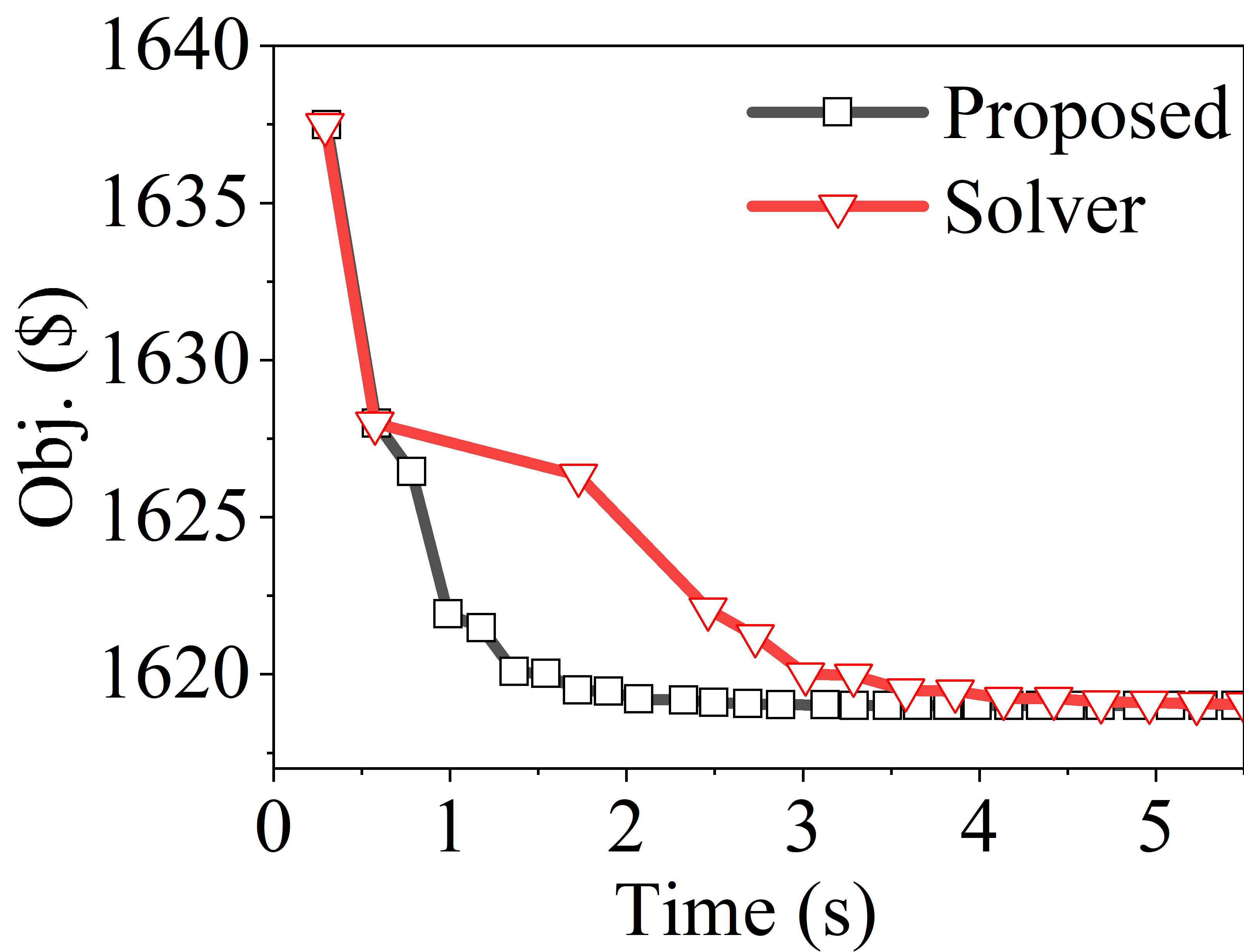}
}

\vspace{0pt}

% ------------------ Row 2 ------------------
\subfloat[Time per iter. (|$\Omega$|=25)]{
    \includegraphics[width=0.30\textwidth]{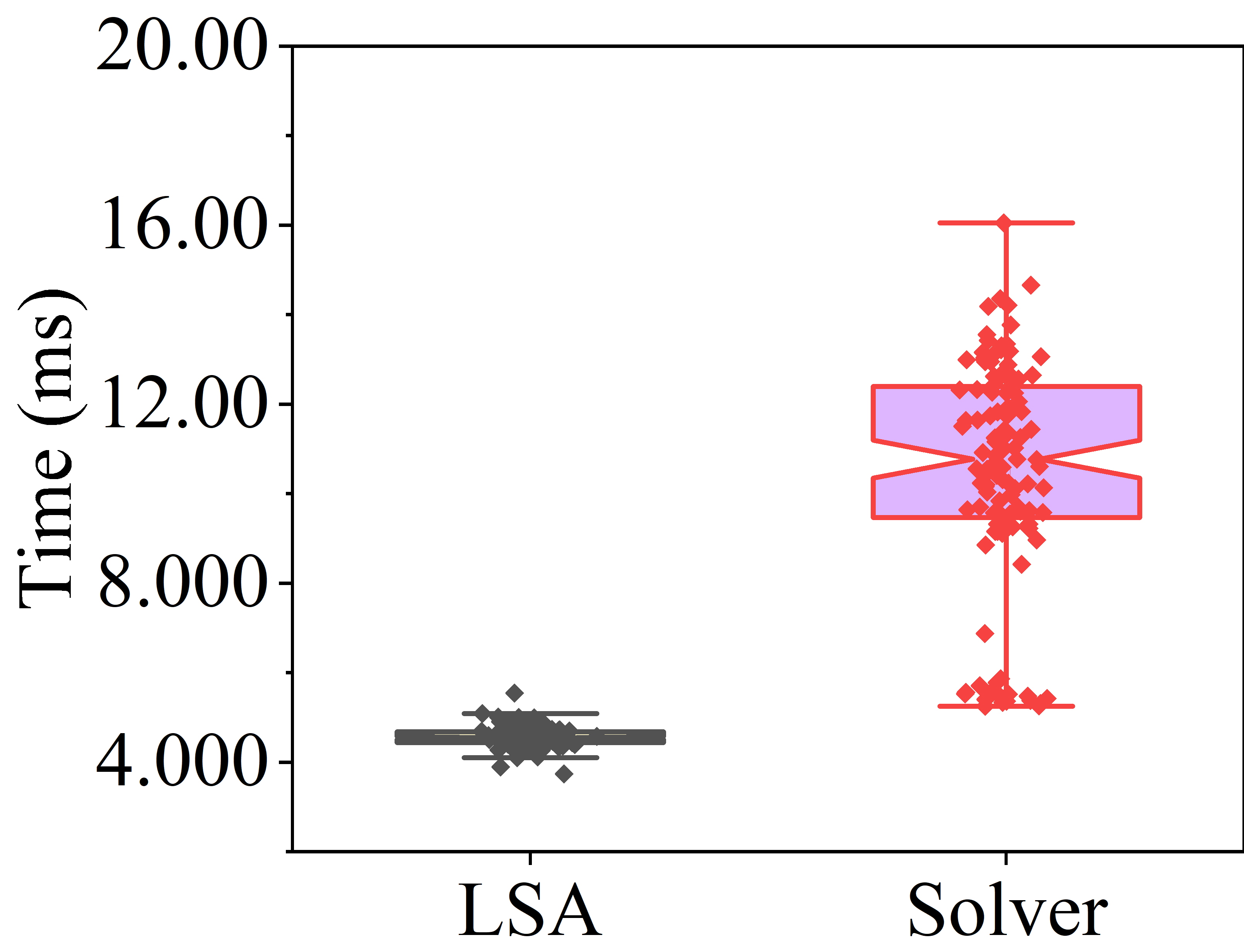}
}
\hfill
\subfloat[Time per iter. (|$\Omega$|=100)]{
    \includegraphics[width=0.30\textwidth]{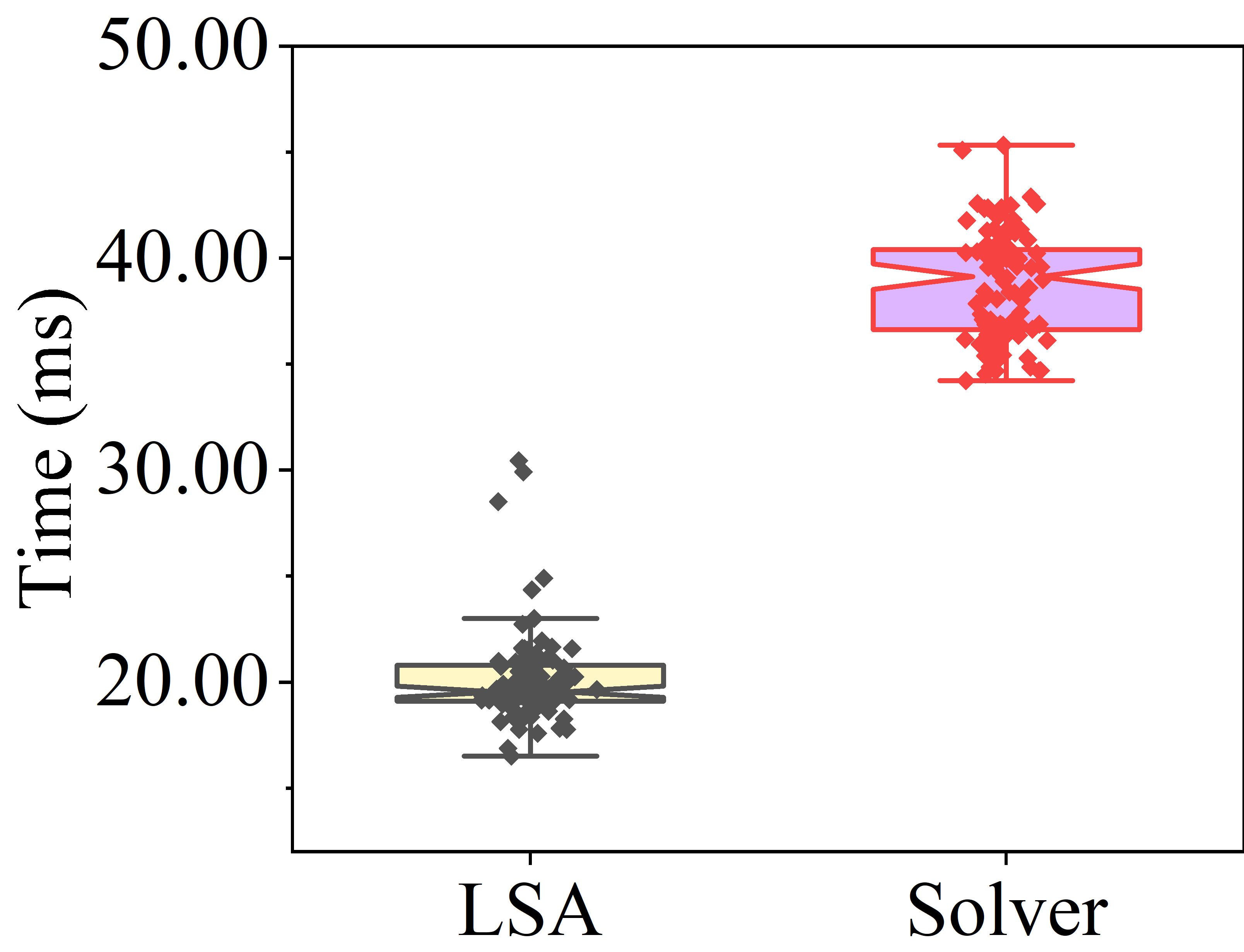}
}
\hfill
\subfloat[Time per iter. (|$\Omega$|=250)]{
    \includegraphics[width=0.30\textwidth]{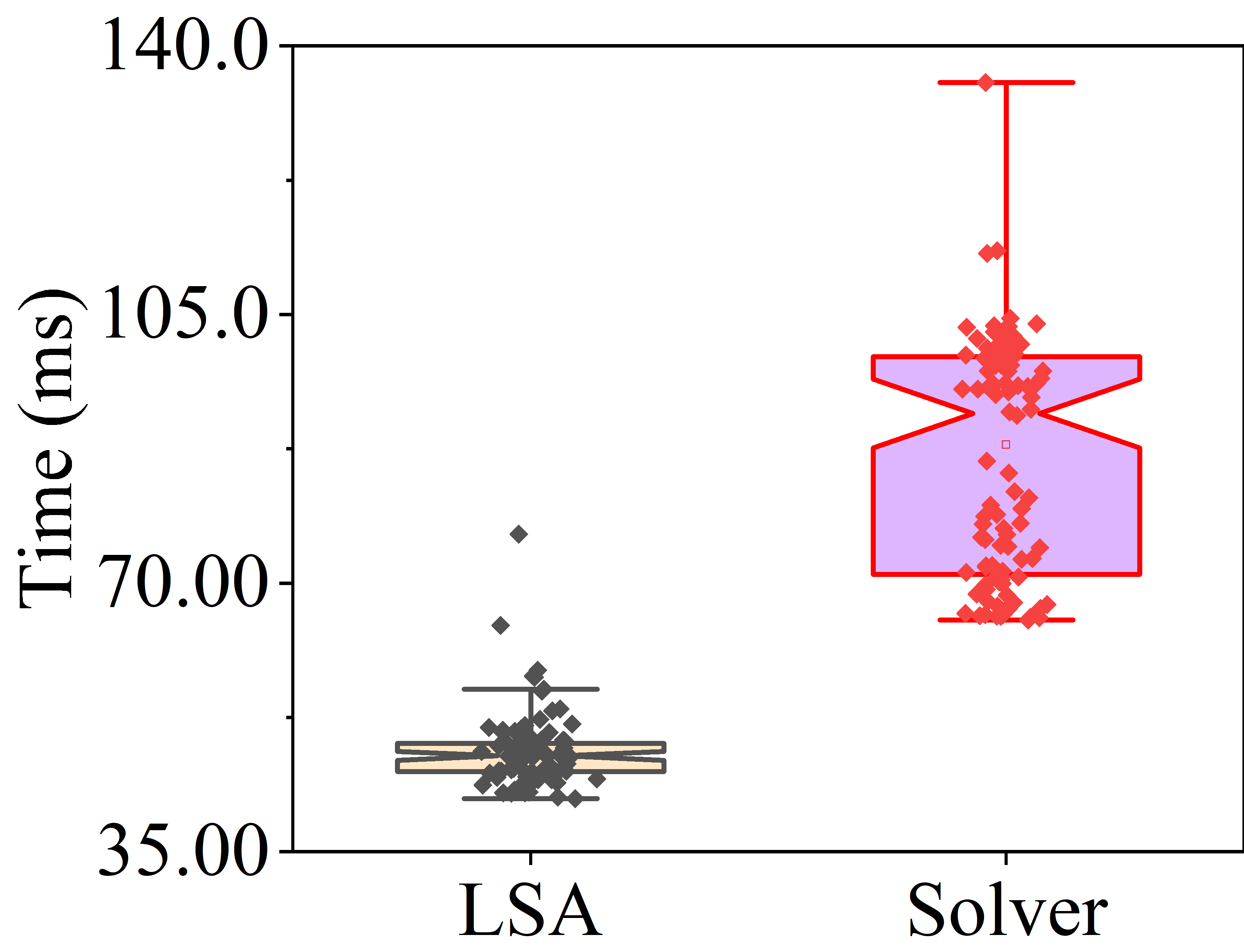}
}

\caption{Performance comparison of the proposed greedy algorithm and the benchmark commercial LP solver under different numbers of price trajectories. The top row illustrates the iteration process and optimality, while the bottom row reports the computational time per iteration (maximum iterations = 50, $\Gamma = 6$).}
\label{fig:lsa_vs_solver_all}
\end{figure*}

The experimental results demonstrate clear computational advantages of the greedy algorithm across all evaluated settings. Specifically, the greedy algorithm achieves remarkable computational acceleration compared to the commercial LP solver: the speedup factor reaches $1.99\times$ under the parameter setting ($\Gamma=6, |\Omega|=100$). A more than $1.98\times$ acceleration is observed in the average per-iteration computational time of the PSM-PAVA algorithm across three day-ahead price scenario settings. Regarding solution quality, Figure~\ref{fig:lsa_vs_solver_all} confirms that the greedy algorithm consistently achieves identical optimal objective values to those obtained by commercial LP solvers in all test instances. The percentage difference between the optimal value of a benchmark method and that of the greedy algorithm is within the preset Gap (Diff), which demonstrates that gains in computational efficiency are achieved without compromising the optimality. These results validate the advantage of leveraging problem-specific structural properties through specialized greedy algorithmic design, rather than employing general-purpose LP solvers for the reformulated two-stage stochastic LP problem.

\subsection{Comparison of Projected Subgradient Method and CC\&G}
\label{Comparison of Projected sudgradient method}

This section benchmarks the PSM-PAVA algorithm against the 
CC\&G algorithm \cite{zeng_solving_2013}, a well-established method 
for solving two-stage robust optimization problems. The CC\&G algorithm 
follows a master--subproblem iterative framework, where candidate 
first-stage decisions are generated in the master problem and 
worst-case realizations are identified through a second-stage 
adversarial subproblem; the resulting cuts are progressively added 
to refine the master problem until convergence. We note that CC\&G 
is naturally applied to the original 2S-ARO formulation~\eqref{eq:2S-ARO},
whereas PSM-PAVA operates on the two-stage stochastic 
LP~\eqref{eq:2S-LP}. To isolate the algorithmic contribution 
from the reformulation benefit, we additionally apply CC\&G directly 
to the same stochastic LP instance; this controlled comparison is 
reported in the ``Stochastic LP problem'' column of 
Table~\ref{tab:psm_ccg_perf}.

To provide a comprehensive comparison, we conduct experiments under different day-ahead price scenario settings. Figure~\ref{PSMvsCCG} illustrates the convergence behavior, solution quality, and computational efficiency of both the PSM-PAVA algorithm and the benchmark CC\&G algorithm under four price scenario sizes: 25, 100, 250, and 500 sampled day-ahead price scenarios. Detailed numerical results are reported in Table~\ref{tab:psm_ccg_perf}. In Figure~\ref{PSMvsCCG}, subfigure (a) illustrates the objective value iteration trajectories of the PSM-PAVA algorithm, and subfigure (b) illustrates those of the CC\&G algorithm. Both algorithms obtain identical optimal solutions, confirming the correctness of the proposed method. In terms of computational efficiency, the PSM-PAVA algorithm shows clear speedups compared with CC\&G. The overall speedup of PSM-PAVA over CC\&G has two distinct sources, 
both quantified in Table~\ref{tab:psm_ccg_perf}. First, the problem 
reformulation in Section~\ref{reformulated formulation} converts the 
nonconvex second-stage problem into the structured LP~\eqref{eq:recourse_given_worst_pv}, 
reducing computation by a factor of $1.7\times$--$3.3\times$ relative to 
solving the original 2S-ARO; this is measured by dividing the ``2S-ARO 
problem'' column by the ``Stochastic LP problem'' column in 
Table~\ref{tab:psm_ccg_perf}, and is consistent with the greedy 
speedup reported in Table~\ref{tab:greedy algorithm_vs_solver}. 
Second, the inner-approximation structure of PSM-PAVA avoids the 
accumulation of second-stage variable copies that inflates the CC\&G 
master problem at each iteration; on the same stochastic LP instance 
this algorithmic advantage contributes a further $37\times$--$101\times$ 
speedup, as shown in the ``Stochastic LP problem'' column. 
We note that part of the per-iteration gap ($\times$2116--$\times$3689) 
is structural: the per-iteration cost of CC\&G grows with the iteration 
count $k$ because each cut appends a full copy of second-stage 
variables, whereas PSM-PAVA maintains a nearly constant per-iteration 
cost by design. Multiplying the two contributions recovers the full 
$122\times$--$204\times$ overall speedup exactly for each scenario size 
tested. The numerically indistinguishable objective differences between PSM-PAVA and CC\&G are mainly due to the stopping tolerance in the outer first-order iterations and do not affect the qualitative computational conclusions.

\begin{figure}
    \centering
    \subfloat[PSM-PAVA algorithm]{%
       \includegraphics[width=0.34\linewidth]{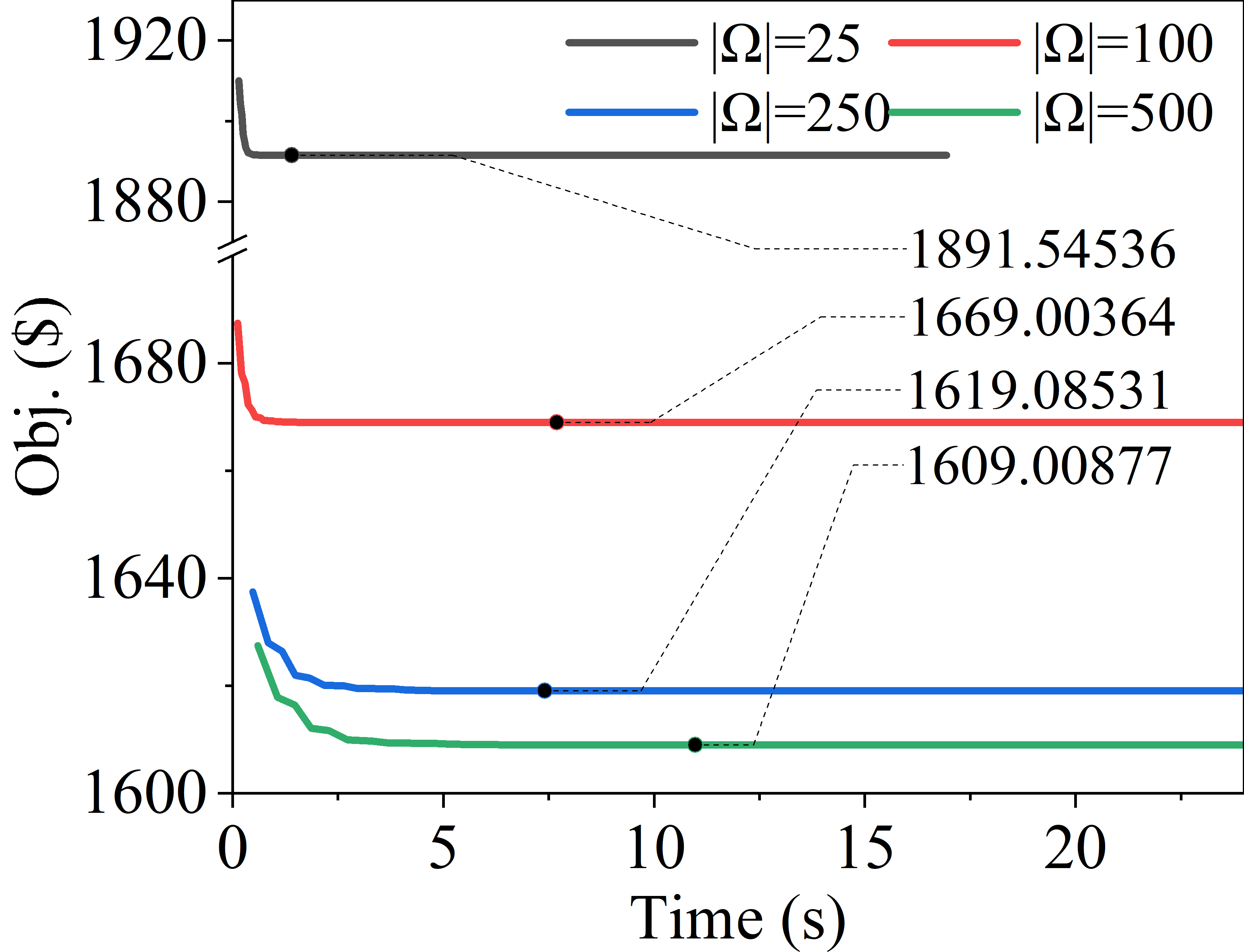}}
       % \hfill
    \subfloat[CC\&G algorithm]{%
    \quad 
    \includegraphics[width=0.34\linewidth]{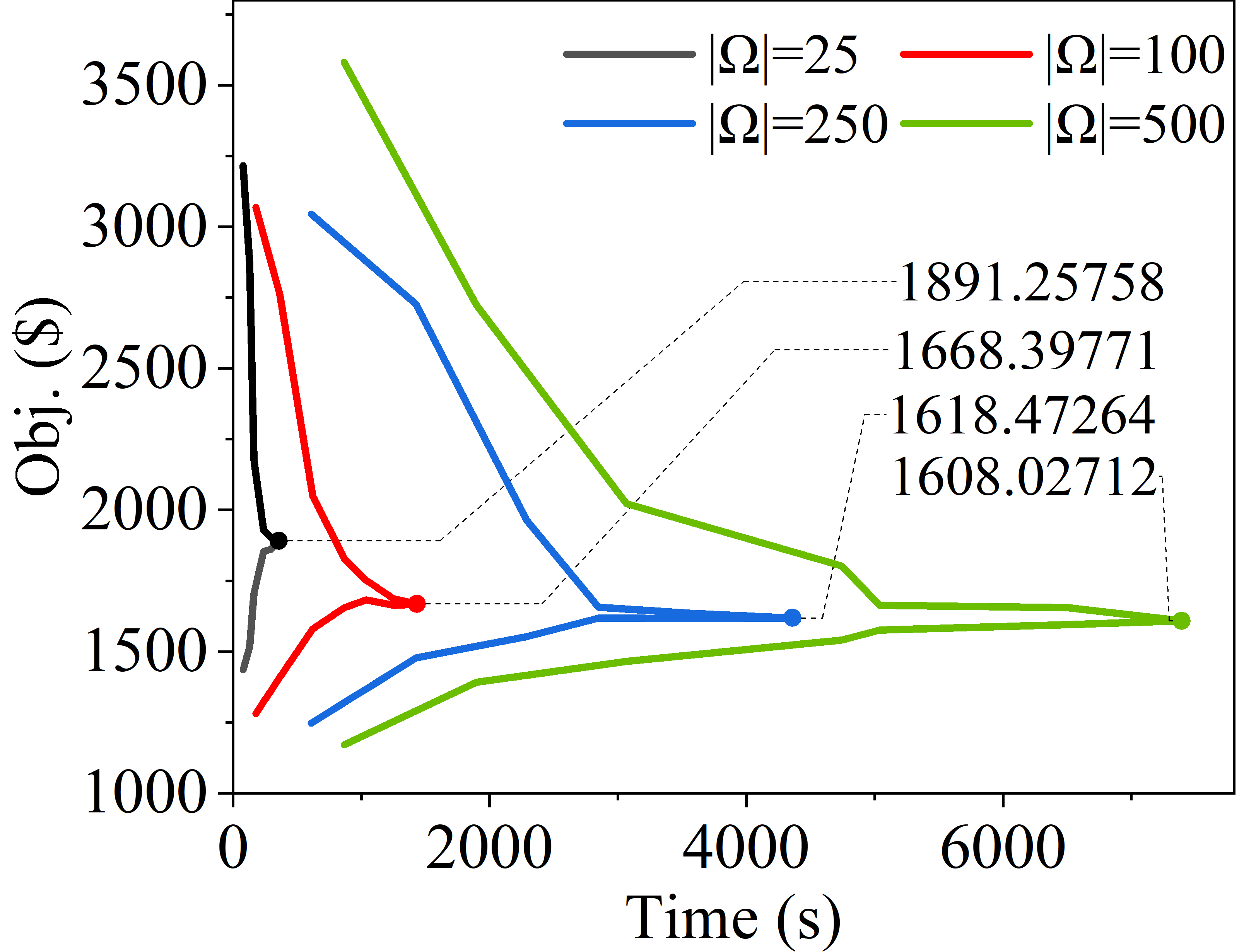}}
       \\
    \caption{Performance comparison of the iteration process, optimality and computational time for the PSM-PAVA and traditional CC\&G algorithm for the 2S-ARO problem (Maximum iteration =600; $\Gamma=6$).}
    \label{PSMvsCCG}
\end{figure}

\begin{table}[t]
  \centering
  \caption{Performance comparison for PSM-PAVA and CC\&G ($\Gamma=6, \varepsilon_{\mathrm{rel}}=10^{-8}$).}
  \label{tab:psm_ccg_perf}
  \scriptsize
  \setlength{\tabcolsep}{2pt}
  \renewcommand{\arraystretch}{1.05}
  \sisetup{detect-weight=true,detect-inline-weight=math,table-number-alignment=center}

  \begin{tabular}{
    S[table-format=3.0]          % Scenario
    S[table-format=5.2]          % 2S-ARO CC&G
    S[table-format=3.2]          % Stochastic LP PSM
    S[table-format=5.2]          % Stochastic LP CC&G
    S[table-format=3.0]          % Speedup
    S[table-format=1.6]          % Per-iter PSM
    S[table-format=4.2]          % Per-iter CC&G
    S[table-format=4.3]          % Obj PSM
    S[table-format=4.3]          % Obj CC&G
    S[table-format=1.4]          % Diff (%)
  }
  \toprule
  \multicolumn{1}{c}{\multirow{2}{*}{Number of scen.}} &
  \multicolumn{1}{c}{2S-ARO problem (s)} &
  \multicolumn{2}{c}{Stochastic LP problem (s)} &
  \multicolumn{1}{c}{\multirow{2}{*}{Speedup}} &
  \multicolumn{2}{c}{Per-iter. (s)} &
  \multicolumn{2}{c}{Obj (\$)} &
  \multicolumn{1}{c}{\multirow{2}{*}{Diff (\%)}} \\
  \cmidrule(lr){3-4}\cmidrule(lr){6-7}\cmidrule(lr){8-9}
  \multicolumn{1}{c}{} &
  \multicolumn{1}{c}{CC\&G} &
  \multicolumn{1}{c}{PSM-PAVA} & \multicolumn{1}{c}{CC\&G} &
  \multicolumn{1}{c}{} &
  \multicolumn{1}{c}{PSM-PAVA} & \multicolumn{1}{c}{CC\&G} &
  \multicolumn{1}{c}{PSM-PAVA} & \multicolumn{1}{c}{CC\&G} &
  \multicolumn{1}{c}{} \\
  \midrule
   25  & 358.04  &
         \multicolumn{1}{c}{\makecell[c]{2.94\\($\times$37)}}   & 109.02  &
         \multicolumn{1}{c}{$\times$122} &
         \multicolumn{1}{c}{\makecell[c]{0.0282\\($\times$2116)}} & 59.67   &
         1891.26 & 1891.26 & 0.0000 \\
  100  & 1432.94 &
         \multicolumn{1}{c}{\makecell[c]{7.26\\($\times$99)}}   & 718.70 &
         \multicolumn{1}{c}{$\times$197} &
         \multicolumn{1}{c}{\makecell[c]{0.0666\\($\times$3172)}} & 238.82  &
         1668.38 & 1668.40 & 0.0012 \\
  250  & 4359.49 &
         \multicolumn{1}{c}{\makecell[c]{29.04\\($\times$88)}}  & 2562.35 &
         \multicolumn{1}{c}{$\times$150} &
         \multicolumn{1}{c}{\makecell[c]{0.269\\($\times$2701)}}  & 726.60  &
         1618.51 & 1618.47 & 0.0025 \\
  500  & 7392.92 &
         \multicolumn{1}{c}{\makecell[c]{36.30\\($\times$101)}}  & 3684.18 &
         \multicolumn{1}{c}{$\times$204} &
         \multicolumn{1}{c}{\makecell[c]{0.334\\($\times$3689)}} & 1232.15 &
         1608.02 & 1608.03 & 0.0006 \\
  \bottomrule
  \end{tabular}
    \vspace{0pt}
  \begin{minipage}{\linewidth}
  \footnotesize
  \textit{Notes:} ``2S-ARO problem (s)'' reports the total computation time of the full 2S-ARO model solved by CC\&G. ``Stochastic LP problem (s)'' reports the computation times when PSM-PAVA and CC\&G are both applied to the same two-stage stochastic LP, isolating algorithmic differences on a common problem instance. ``Speedup'' is the ratio of the 2S-ARO computation time of CC\&G to the stochastic-LP computation time of PSM-PAVA.
  \end{minipage}
\end{table}

\subsection{Offering Strategy Results}
\label{Offering strategy results}

The numerical results provided in this section serve three purposes: First, to clarify how the VPP determines its offering strategy in the day-ahead market at different hours. Second, 
to verify the arbitrage behavior of the energy storage unit within VPP under the algorithm of PSM-PAVA. 
Third,
to analyze the impact of the risk strategy on the power traded in the day-ahead market and the actual utility obtained by the VPP. To
accomplish this, two values of parameters $\Gamma =6,12$, i.e.,
two different risk strategies, are considered and compared. 

The day-ahead market offering strategies of VPP at time slots 1, 11, 12, and 22 under two conservatism levels ($\Gamma = 6$ and $\Gamma = 12$) is presents in Figure~\ref{Price-quantity-Pairs}. The temporal analysis reveals distinct offering behaviors: at $\Gamma = 6$, the VPP submits purchase bids within $[1.0426, 1.0463]$~MW during time slot 1, transitioning to sales offers in the range $[0.0045, 0.0055]$~MW at time slot 12. This offering behavior corresponds to the exploitation of price differentials for VPP captured by the Markov chain transition probabilities, whereby the expected price at time slot 12 exceeds that of time slot 1. As the conservatism parameter $\Gamma$ decreases, the offering curves consistently shift leftward and VPP becomes more willing to sell and less inclined to purchase power, reflecting its increased risk tolerance. The VPP behaves exactly like an arbitrageur: purchasing cheap during off-peak hours and selling high when prices increase. Specifically, relaxing the conservatism constraint from $\Gamma = 12$ to $\Gamma = 6$ yields increased sales volumes during high-price periods and reduced procurement during low-price intervals. Reduced conservatism parameters yield higher offering quantities as the VPP more actively exploits intertemporal price differentials. 
% This behavior emerges because relaxed uncertainty bounds enable the VPP to leverage its storage capacity more effectively for temporal arbitrage.

\begin{figure}[!t]
\centering
\includegraphics[width=\linewidth]{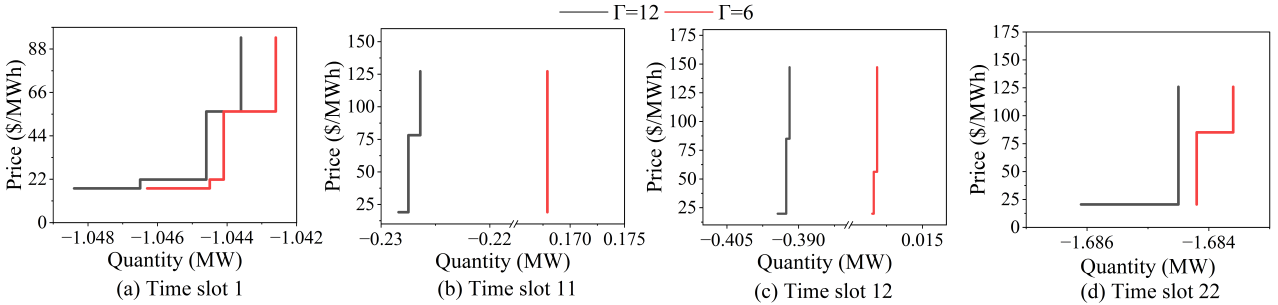} % 请将fig1替换为实际的图像文件名
\caption{Offering curves in the day-ahead market for time slots 1, 11, 12, and 22.}
\label{Price-quantity-Pairs}
\end{figure}

Figure~\ref{ES-Price} depicts the energy storage unit power exchange in the day-ahead market throughout the operational day under six day-ahead price scenarios selected from the 25 sampled price trajectories. 
The results demonstrate that the energy storage unit effectively exploits day-ahead price fluctuations for arbitrage operations, charging during low-price periods and discharging when prices are relatively higher.

\begin{figure}[!t]
\centering
\includegraphics[width=6.2 in]{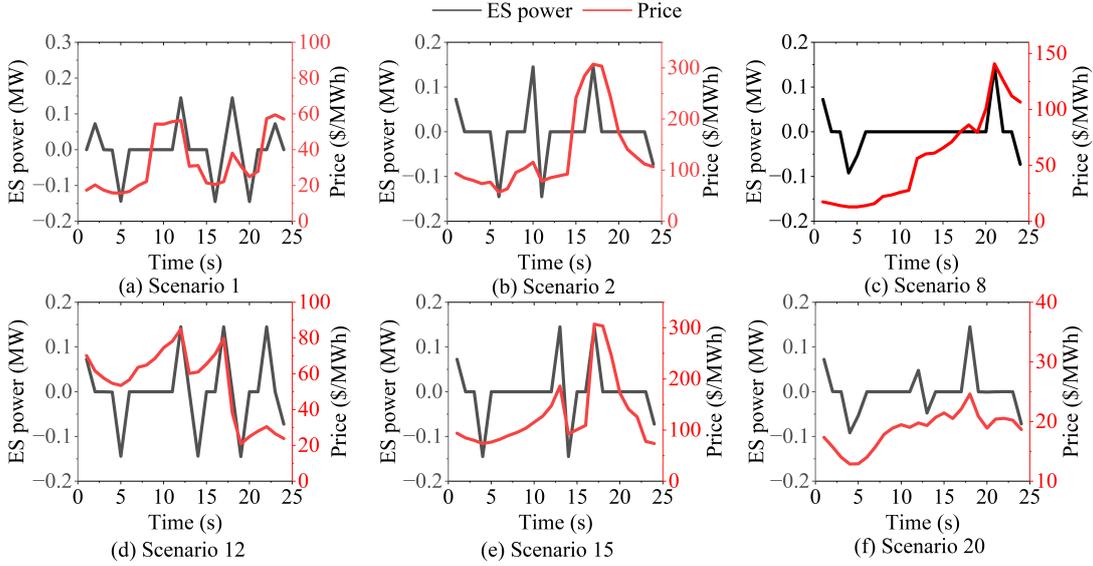} % 请将fig1替换为实际的图像文件名
\caption{Energy storage unit dispatch decisions within the VPP under six price scenarios.}
\label{ES-Price}
\end{figure}

\subsection{Out-of-sample Analysis}
\label{Out-of-sample Analysis}

To further evaluate the performance of the proposed day-ahead offering strategy, out-of-sample analyses are conducted. The evaluation employs scenario sets ranging from 25 to 10,000 day-ahead price realizations. The expected profit achieved by the VPP through the PSM-PAVA algorithm is computed for each scenario realization. Figure~\ref{Iteration} reports the average profit and its variance across strategies as $|\Omega|$ increases, where each data point is averaged over five independently drawn scenario sets of the same size. The coefficient of variation falls below 0.5\% for $|\Omega| \geq 500$, confirming that 500 sampled scenarios yield statistically stable solutions. The iterative objective values also stabilize at this point, indicating that near-optimal convergence is attained with 500 price scenarios. The analysis reveals that the average profit at $\Gamma=6$ is higher than that at 12.

\begin{figure}[!t]
\centering
\includegraphics[width=3.8 in]{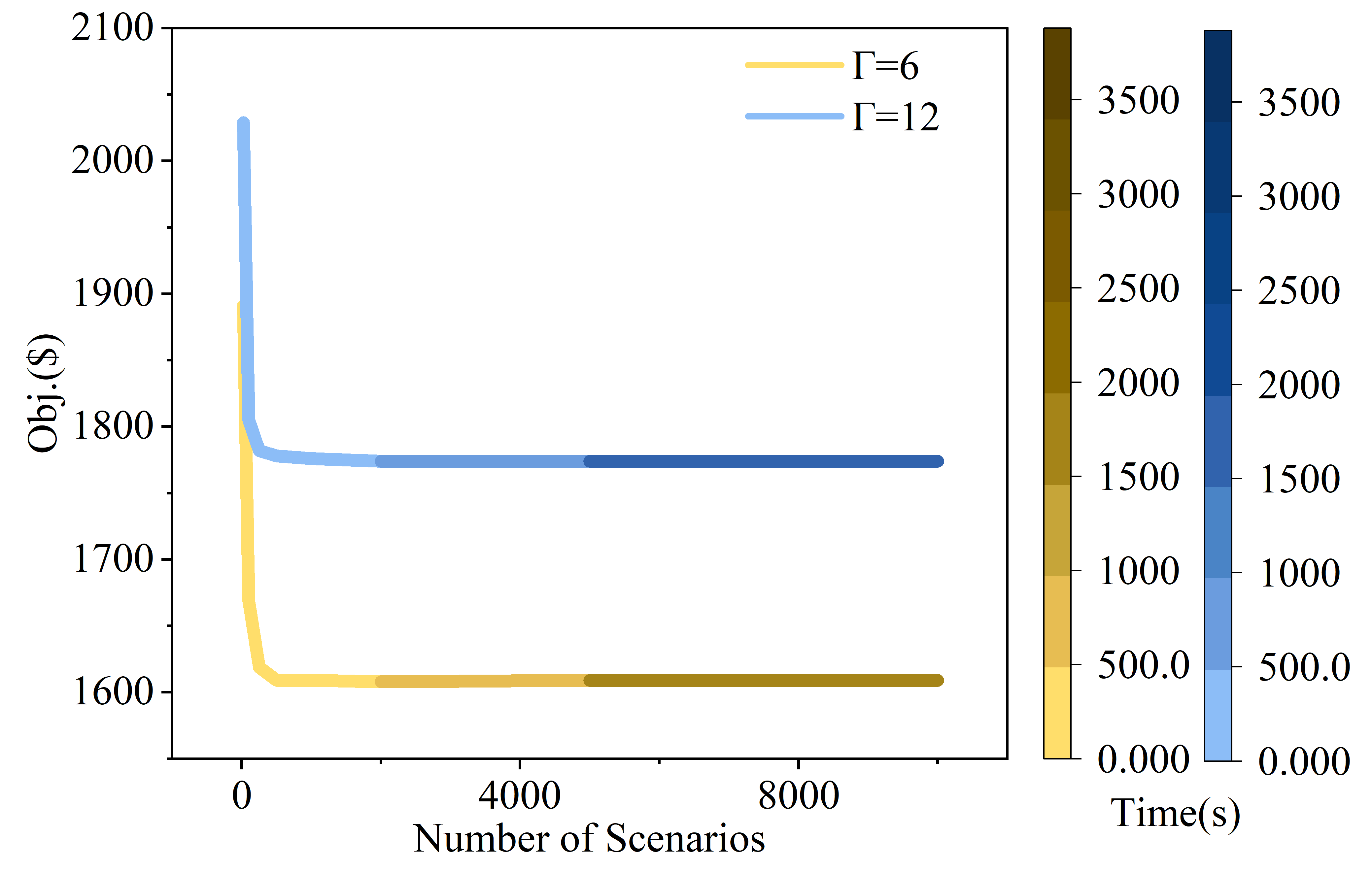} % 请将fig1替换为实际的图像文件名
\caption{Evolution of the average profit with the number of sampled price scenarios.}
\label{Iteration}
\end{figure}

\begin{figure}[!t]
\centering
\includegraphics[width=5.2 in]{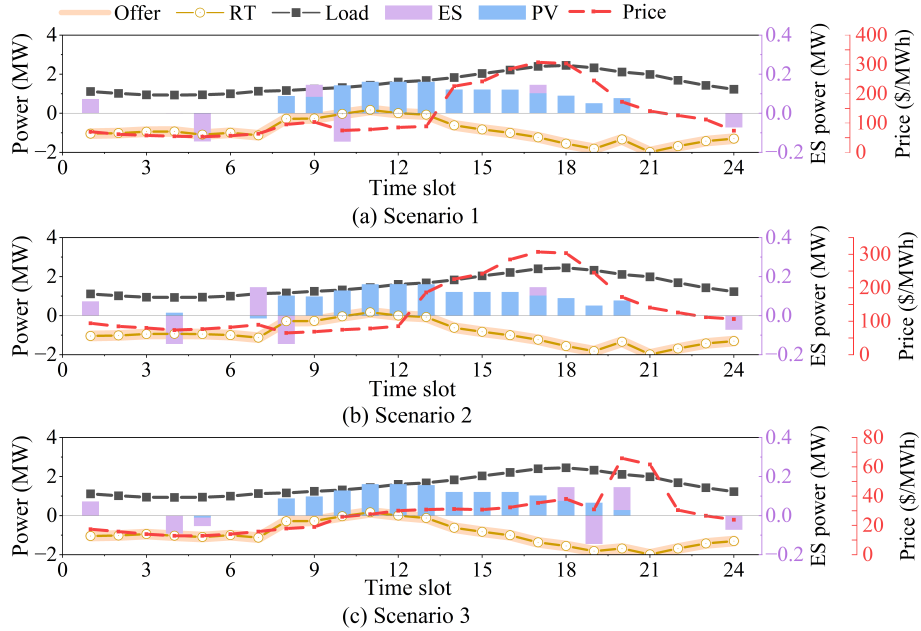} % 请将fig1替换为实际的图像文件名
\caption{Economic dispatch decisions of VPP under three price scenarios (``RT'': the operational power).}
\label{ClearED}
\end{figure}

To evaluate the practical performance of the PSM-PAVA algorithm for the two-stage stochastic LP problem, an operational economic dispatch case study is implemented to simulate real-world conditions. Historical day-ahead price and PV generation data are used to represent uncertainty realizations. Committed offers under the day-ahead offering strategy are settled at the day-ahead price. The economic dispatch is performed using these offer quantities alongside realized DER outputs. Figure~\ref{ClearED} presents the results of the economic dispatch under three random day-ahead price trajectories and PV uncertainty realizations, all collected from historical data. The results show that PV generation is partially curtailed to meet the commitment offer. No power mismatch occurs throughout the time horizon, and the energy storage unit dispatch arbitrage aligns with the prevailing day-ahead price fluctuations. The offering results also indicate that there is no arbitrage behavior for VPP to arbitrage from the two-settlement markets, which aligns well with the market rule of Remark \ref{remark real time price}. The energy storage unit operates like a financial arbitrageur during the operational day, which aligns with the common wisdom that the energy storage units are most effective in cost reduction during real-time operations.

% \vspace{-0.2in}
\section{Conclusion}
\label{conclusion}

% This study addresses the two-stage stochastic ARO problem for VPP participation in the day-ahead electricity market. First, under the adopted modeling assumptions, the original two-stage stochastic ARO problem can be reformulated as a two-stage stochastic LP problem, thereby reducing computational complexity. Second, the proposed greedy algorithm constructs an exact recourse oracle for the second-stage LP problem by exploiting the nested resource allocation structure. Third, extensive numerical experiments demonstrate the applicability, computational efficiency, and solution quality of the PSM-PAVA algorithm. The PSM-PAVA algorithm achieves two orders of magnitude total speedup over CC\&G, decomposing into around a $2\times$ gain from problem reformulation and around a $70\times$ gain from the inner-approximation algorithmic design. The proposed second-stage reformulation and the PSM-PAVA algorithm facilitate the efficient solution of the VPP day-ahead offering problem with clearly shorter computational times.

This study addresses the two-stage stochastic ARO problem for VPP participation in the day-ahead electricity market. Under the adopted modeling assumptions, we show that the original two-stage stochastic ARO model can be equivalently reformulated as a two-stage stochastic LP problem, which substantially improves tractability while preserving the models of the offering problem. In particular, the reformulation converts the original nonconvex robust second-stage problem into a structured recourse problem with a nested resource allocation structure amenable to efficient algorithmic development.

Building on the reformulation, we develop an exact greedy oracle for the second-stage LP problem and integrate it into a projected subgradient framework for solving the overall offering problem. The role of the greedy oracle is not only to evaluate the scenario-wise recourse function efficiently, but also to support the outer algorithm with the exact recourse information needed for subgradient construction. This leads to the proposed PSM-PAVA algorithm, which avoids the scalability issue of CC\&G-type approaches and instead maintains a lightweight first-order iteration structure.

Extensive numerical experiments based on real-world data demonstrate the practical value of the proposed framework. The greedy oracle consistently attains the same solution quality as a commercial LP solver while reducing the computational cost of the second-stage problem. At the full problem level, the PSM-PAVA algorithm achieves about two orders of magnitude total speedup over CC\&G, which can be interpreted as the combined effect of two sources: approximately a $2\times$ gain from the second-stage reformulation itself and approximately a $70\times$ gain from the inner-approximated first-order algorithmic design. These results show that the proposed framework improves computational efficiency without sacrificing solution quality.

Overall, the paper provides an efficient solution algorithmic framework for VPP day-ahead offering under coupled price and renewable uncertainty. The main contribution is not only the reformulation of the original two-stage stochastic ARO problem, but also the development of an exact and computationally efficient recourse oracle that makes repeated scenario-wise evaluation practical within a large-scale first-order algorithm. The proposed second-stage reformulation and the PSM-PAVA algorithm facilitate the efficient solution of the VPP day-ahead offering problem by exploiting the problem structure, and it provide a promising path for scalable market participation models of DER aggregations.

\section*{Acknowledgements}
This work was supported by the U.S. Department of Energy Solar Energy Technologies Office under Grant DE-EE0011374.

\clearpage

\appendix 
\numberwithin{equation}{section}

\section{Proof of Proposition \ref{NoCDischarge}} \label{sect:app:NSCD}

\begin{proof}
Suppose there is an optimal solution to the second-stage problem \eqref{eq:cost-to-go function} without constraint \eqref{eq:ES_cc} with positive $\bigl(p_{t}^{\mathrm{ES,ch}}\bigr)^*$ and $\bigl(p_{t}^{\mathrm{ES,dis}}\bigr)^*$ and the associated SOC $e^*$. Since both charging and discharging powers are positive, there exists $\epsilon > 0$ such that $p_{t}^{\mathrm{ES,ch}} = \bigl(p_{t}^{\mathrm{ES,ch}}\bigr)^* - \epsilon$ and $p_{t}^{\mathrm{ES,dis}} = \bigl(p_{t}^{\mathrm{ES,dis}}\bigr)^* - \eta^{\mathrm{ch}}\eta^{\mathrm{dis}}\epsilon$ are still nonnegative, and the original $e^*$ is preserved, which in turn preserves the feasibility of the updated solution. For the new feasible solution, the energy storage unit discharging power $p_{t}^{\mathrm{ES,dis}}$ decreases, while the net power output $P_{t}^\ES$ increases. By checking the objective function of the second-stage problem \eqref{eq:Pi_RT}, \eqref{eq:RT_cost}, we see it is nondecreasing in $p_{t}^{\mathrm{ES,dis}}$ and decreasing in $P_{t}^\ES$; therefore, the new solution has a strictly lower objective value. We have thus shown that a solution with positive charging and discharging power $p_{t}^{\mathrm{ES,ch}}$ and $p_{t}^{\mathrm{ES,dis}}$ cannot be an optimal solution.
\end{proof}

\section{Proof of Proposition \ref{DecouplingCons}}\label{sect:app:DPC}

\begin{proof}
Fix a first-stage decision $p^{\DA}\in\mathcal X$ and a price scenario $\omega\in\Omega$.
By Proposition~\eqref{NoCDischarge}, the complementarity constraint
\eqref{eq:ES_cc} can be dropped without affecting optimality, so the second-stage problem can be analyzed through the reformulated problem \eqref{eq:reformulated_cost_to_go}.

% through the reformulated objective
% \eqref{eq:imbalance cost reformulated} together with constraints
% \eqref{eq:pv_bound} and \eqref{eq:ES_bounds}.

We first show that, for any fixed realization $\hat p^{\PV}\in\mathcal U$,
the optimal PV dispatch decision does not involve curtailment. From \eqref{eq:imbalance cost reformulated}, each affine piece of
$f_{t,\omega}(\Delta e_t, p_t^\PV, p_t^\DA)$ has slope with respect to $p_t^{\PV}$ equal to either $\tilde c_{t,\omega}^\PV$ or $\utilde c_{t,\omega}^\PV$.
By the assumption \eqref{eq:kappa_cond_prop}:
\begin{equation}
\kappa\le \lambda_{t,\omega}^{\DA}-c^{\PV},
\qquad \forall t\in\mathcal T,
\end{equation}
which implies
\begin{equation}
\label{eq:kappa_cond_prop_1}
\utilde c_{t,\omega}^\PV \leq
\tilde c_{t,\omega}^\PV \leq 0,  \qquad \forall t\in\mathcal T.
\end{equation}

Hence, every affine piece of $f_{t,\omega}(\Delta e_t, p_t^\PV, p_t^\DA)$ is nonincreasing in
$p_t^{\PV}$, and so is their pointwise maximum.
Since the only constraint involving $p_t^{\PV}$ is the box constraint \eqref{eq:pv_bound}, it follows that, for any realized $\hat p_t^{\PV}$ and feasible $\Delta e_t$, the optimal PV dispatch is achieved at its upper bound:
\begin{equation} \label{eq:appDPC_3}
p_t^{\PV,*}(\omega)=\hat p_t^{\PV}(\omega),\qquad \forall t\in\mathcal T.
\end{equation}

After eliminating the PV dispatch variable, for a given feasible $\Delta e_t$, the outer maximization problem becomes a linear optimization problem over $\mathcal U$ whose coefficients are induced by the active affine pieces in \eqref{eq:imbalance cost reformulated}. For each $t\in\mathcal{T}$, the corresponding effective coefficient is either $\tilde c_{t,\omega}^{\PV}$ or $\utilde c_{t,\omega}^{\PV}$. We next show that the ordering of these effective coefficients is independent of $\Delta e_t$.

Indeed, let $t, t' \in \mathcal{T}$ be such that $\lambda_{t,\omega}^\DA > \lambda_{t',\omega}^\DA$. By \eqref{eq:kappa_cond_prop}, we have $\lambda_{t,\omega}^\DA - \lambda_{t',\omega}^\DA \ge 2\kappa$, which means
\begin{equation}
    \tilde c_{t,\omega}^{\PV}-\utilde c_{t',\omega}^{\PV} = \bigl(c^{\PV}-\lambda_{t,\omega}^{\DA}+\kappa\bigr) - \bigl(c^{\PV}-\lambda_{t',\omega}^{\DA}-\kappa\bigr) = -(\lambda_{t,\omega}^{\DA}-\lambda_{t',\omega}^{\DA})+2\kappa \le 0.
\end{equation}
Therefore, every admissible effective coefficient at time $t$ is no larger than every admissible effective coefficient at time $t'$ whenever $\lambda_{t,\omega}^{\DA}\ge\lambda_{t',\omega}^{\DA}$. For time periods with equal day-ahead prices, the corresponding effective coefficients have the same ordering regardless of the active affine piece. It follows that the ranking of the effective coefficients is consistent with that of the negative day-ahead price $-\lambda_{t,\omega}^{\DA}$ and is independent of $\Delta e_t$.     

Since the PV uncertainty range is constant across time, the resulting outer problem is a continuous knapsack problem whose optimal solution depends solely on the ordering of its coefficients. Therefore, one worst-case PV availability profile can be chosen as an optimal solution of the surrogate problem as
\begin{equation} \label{eq:proof_linear_outer}
    \hat p^{\PV,*}(\omega)\in \arg\max_{\hat p^{\PV}\in\mathcal U} \sum_{t\in\mathcal T}-\lambda_{t,\omega}^{\DA}\hat p_t^{\PV}.
\end{equation}
Combining this with \eqref{eq:appDPC_3}, we conclude that the optimal PV dispatch in the second-stage problem satisfies
\begin{equation}
    p_t^{\PV,*}(\omega)=\hat p_t^{\PV,*}(\omega), \qquad \forall t\in\mathcal{T},
\end{equation}
which prove part (i).

Finally, once $p_t^{\PV,*}(\omega)$ is fixed, all terms involving PV availability
become parameters, and the second-stage problem reduces to an optimization over the
energy storage decision variables only, which is exactly the problem \eqref{eq:recourse_given_worst_pv}.
This proves part (ii).
\end{proof}

\section{Proof of Proposition \ref{thm:GDfeasibility}} \label{app:C}
\begin{proof}
    We prove the two inequality conditions \eqref{eq:SOCineq} and \eqref{eq:ubvineq} by induction. At initialization, 
    \begin{equation}
        e_{t}^{(0)} = e_0 + \sum\nolimits_{j=1}^t \Delta e_j^{(0)}
        = \sum\nolimits_{j=1}^t (\bar e_j - \ubar{e}_{j-1}) + e_0 
        = \bar{e}_t + \sum\nolimits_{j=1}^{t-1} (\bar e_j - \ubar e_j) \ge \ubar{e}_t.
    \end{equation}
    In addition, it is also easy to derive
    \begin{equation}
        e_t^{(0)} - \bar e_t = \sum\nolimits_{j=1}^{t-1} (\bar e_j - \ubar e_j) \quad \text{ and } \quad e_t^{(0)} - \ubar e_t = \sum\nolimits_{j=1}^{t} (\bar e_j - \ubar e_j).
    \end{equation}
    Therefore, for any $r<t$, the following inequality follows:
    \begin{equation}
        [e_r^{(0)} - \bar{e}_r]_+ = \sum\nolimits_{j = 1}^{r-1} (\bar e_j - \ubar e_j) \le \sum\nolimits_{j = 1}^{t} (\bar e_j - \ubar e_j) = e^{(0)}_t - \ubar e_t.
    \end{equation}
    Taking the maximum over $r < t$ gives
    \begin{equation}
        M_t^{(0)} \le e_t^{(0)} - \ubar e_t.
    \end{equation}

    We now assume the two inequalities \eqref{eq:SOCineq} and \eqref{eq:ubvineq} hold at iteration $k$, and consider iteration $k+1$. First of all, line \ref{line:na} does not change $\Delta e^{(k)}$, so the inequalities remain unchanged. Otherwise, Algorithm \ref{alg:local_search_algorithm} decreases exactly one component $\Delta e_t$ by some $\delta \ge 0$, and consequently, the induced SOC changes as 
    \begin{equation}
        e_j^{(k+1)} = 
        \begin{cases}
            \Delta e_j^{(k)} & \text{if } j < t \\
            \Delta e_j^{(k)} - \delta & \text{if } j \ge t
        \end{cases}
    \end{equation}
    Thus, for $j < t$, both the lower bound slack $e_j - \ubar{e}_j$ and the upper bound violation $[e_j - \bar{e}_j]_+$ are unchanged; for $j \ge t$, the lower bound slack decreases by $\delta$, while the upper bound violation decreases as much or settles at zero. Therefore, the only way the inequality \label{eq:Mineq} could fail is when an upper bound violation $r < t$ stays unchanged while the later lower bound slack at some $j \ge t$ is reduced. To prevent this from happening, $\delta$ should be chosen so that
    \begin{equation} \label{eq:d2cond}
        \delta \le \min_{t \le j \le T} \{e_j^{(k)} - \ubar e_j\} - \max_{1 \le j \le t-1} \{[e_j^{(k)} - \bar e_j]_+\}
    \end{equation}

    We consider two update cases, one in line \ref{line:update}, the other in line \ref{line:update_negslope}. First, we consider the update in line \ref{line:update}. Here, the line defines exactly $\delta_2 = \min_{t \le j \le T} \{e_j^{(k)} - \ubar e_j\} - \max_{1 \le j \le t-1} \{[e_j^{(k)} - \bar e_j]_+\}$, which is the right-hand side of \eqref{eq:d2cond}. Since $\delta^* \le \delta_2$, condition \eqref{eq:ubvineq} is preserved. In addition, the SOC lower bound \eqref{eq:SOCineq} is also preserved since $\delta_2 \le \min_{t \le j \le T} \{e_j^{(k)} - \ubar e_j\}$.
    
    Now consider the update in line \ref{line:update_negslope}, and suppose $\tau \coloneqq \min \{j: e_j^{(k)} - \bar e_j > 0 \}$. This means $[e_j^{(k)} - \bar e_j]_+ = 0$ for all $j < \tau$. The algorithm selects some $\tau' < \tau$ and defines $\delta_2$ as the smallest lower bound slack between time steps $\tau'$ and $\tau-1$, and $\delta_3$ as the upper bound violation at time step $\tau$ according to line \ref{line:d1d2d3}. Since $\delta^* \le \delta_2$, the SOC lower bounds are preserved for all time steps between $\tau'$ and $\tau-1$. For time steps $t \ge \tau$, the induction hypothesis gives
    \begin{equation}
        e_t^{(k)} - \ubar e_t \ge M_t^{(k)} \ge [e_\tau^{(k)} - \bar e_\tau]_+ = e_\tau^{(k)} - \bar e_\tau = \delta_3.
    \end{equation}
    Since $\delta^* \le \delta_3$, every time step $t \ge \tau$ also remains above its lower bound after the update. Regarding the upper bound violations, since there are no violations before $\tau$, and consistent SOC decrease does not affect \eqref{eq:ubvineq}. 

    Thus, both inequalities \eqref{eq:SOCineq} and \eqref{eq:ubvineq} are preserved for every iteration $k$.

    Now we show the algorithm must reach feasibility eventually. Since \eqref{eq:SOCineq} always holds, the potential infeasibility has to come from the upper bound violation. The algorithm returns either when the solution is feasible (line \ref{line:early_return}), or when $\mathcal{A}$ is empty. We may ignore the removal of active segment in line \ref{line:d1remove1} and line \ref{line:d1remove2}, as they never remove the last active segment from a time step (there is no further breakpoint). We are left with three other segment removal steps in lines \ref{line:d2remove1}, \ref{line:d2remove2}, and \ref{line:d2remove3}. Line \eqref{line:d2remove1} does not remove the segments of a time step with upper bound violation. Line \ref{line:d2remove2} removes active segments for all time steps before one with binding lower bound constraint. By \eqref{eq:ubvineq}, we know these do not include time steps with upper bound violations. Finally, line \eqref{line:d2remove3} does not remove the active segments in time step $\utilde \tau$, who has upper bound violation. In summary, the algorithm does not return when there is upper bound violation since $\mathcal{A} \ne \varnothing$ in this case, and since the algorithm returns in finite iterations, it returns with a feasible solution.
\end{proof}

\section{Proof of Proposition \ref{thm:GDoptimality}} \label{app:D}
\begin{proof}
    We prove optimality via an equivalent reformulation as a separable resource allocation problem with nested cumulative constraints. Specifically, by defining $d \coloneqq \Delta e^{(0)} - \Delta e$, Problem \eqref{eq:2ndcompact} can be reformulated as
    \begin{subequations}
    \begin{align*}
        \max_{d} \quad & \sum_{t=1}^T \psi_t(d_t) \\
        \text{s.t.} \quad & d_t \ge 0, && t\in \mathcal{T} \\
        & \sum_{j=0}^{t-1} (\bar e_j - \ubar e_j) \le \sum_{j=1}^t d_j \le \sum_{j=0}^t (\bar e_j - \ubar e_j), &&  t \in \mathcal{T} \setminus\{T\}, \\
        & \sum_{j=1}^T d_j = \sum_{j=0}^{T-1} (\bar e_j - \ubar e_j),
    \end{align*}
    \end{subequations}
    where 
    \begin{equation}
        \psi_t(d_t) = f_{t,\omega}(\Delta e_t^{(0)}) -f_{t,\omega} (\Delta e_t^{(0)} - d_t).
    \end{equation}
    Since each $f_{t,\omega}$ is convex piecewise linear, each $\psi_t$ is concave piecewise linear. Moreover, increasing $d_t$ by $\delta$ is exactly equivalent to decreasing $\Delta e_t$ by $\delta$. Therefore, after the reformulation, Problem \eqref{eq:2ndcompact} becomes a standard separable resource allocation problem with nested cumulative constraints in the newly defined $d$ variable space. The algorithm can then be viewed as a specialized implementation of the classical greedy marginal allocation paradigm for this problem class \cite{federgruen1986greedy}, where the algorithm chooses the best available marginal segment and increases the corresponding $d_t$ until one of the following three happens:
    \begin{itemize}
        \item The current affine segment is exhausted; 
        \item An upper bound constraint becomes binding; and 
        \item A lower bound constraint becomes binding. 
    \end{itemize}
    The three scenarios correspond to the roles of $\delta_1, \delta_2$, and $\delta_3$ in Algorithm \ref{alg:local_search_algorithm}. 
\end{proof}

\newpage
\bibliographystyle{elsarticle-harv} 
\bibliography{Ref/references.bib} 





% \bibitem[Lamport(1994)]{lamport94}
%   \textit{\LaTeX: a document preparation system},
%   Addison Wesley, Massachusetts,

% \end{thebibliography}

\end{document}

%% file: section_formulation_revised_MWQ.tex
\section{Problem Statement and Optimization Formulation}
\label{sect:formulation}
In this section, we describe the day-ahead offering problem of a price-taking VPP and present the corresponding two-stage stochastic ARO formulation. 

% \subsection{Problem Statement}
\subsection{Problem Setup and Market Setting}

\label{subsec:decision_model}
We model a utility-operated VPP that aggregates a solar PV unit, an energy storage unit, and an aggregated load over a day-ahead horizon $\mathcal{T} = \{1, 2, \ldots, T \}$ (with $T = 24$ in our case studies). Their power outputs at time $t \in \mathcal{T}$ are denoted by $p_t^{\mathrm{PV}}$, $p_t^{\mathrm{ES}}$, and $p_t^{\ell}$, respectively\footnote{We model the system load as an aggregated demand. Aggregated loads exhibit stable and well‑forecasted temporal patterns in typical operational horizons, so the forecasted load profile is used as deterministic input. While load uncertainty can be modeled through budgeted uncertainty sets, doing so would preserve the same robust optimization structure adopted in this study and would not change the qualitative conclusions.}. 
% The power output of the VPP is therefore 
% \begin{equation}
%     p_t = p_t^{\mathrm{PV}} + p_t^{\mathrm{ES}} + p_t^{\ell}.
% \end{equation}
The VPP participates in the day-ahead electricity market, in which it submits a stepwise offer curve (price-quantity pairs) for each hour $ t \in \mathcal{T}$. The system operator clears the day-ahead market after receiving offer curves from all market participants and determines the realized day-ahead clearing price and the accepted (committed) day-ahead quantity. 

During the operational (so-called real-time in the two-settlement market) stage, the VPP dispatches its resources (PV and storage) to realize the committed quantity and settles any power imbalance with the system operator. The profit of VPP, therefore, is determined by (i) day-ahead settlement revenue, (ii) real-time imbalance settlement, and (iii) DER operating costs.
\begin{remark} 
\label{remark real time price}
In U.S. ISO/RTO two-settlement markets, real-time settlements apply to the deviation between metered real-time quantities and the day-ahead commitment, priced at the real-time LMP (see, e.g., MISO BPM-005 in \cite{miso_bpm005_2025}; PJM settlements documentation in \cite{pjm_customer_guide_billing_2025}). In addition, to eliminate the arbitrage behavior from two-settlement markets, some markets impose penalties for uninstructed deviations beyond tolerance bands (e.g., CAISO Tariff Section 11 on Uninstructed Deviation Penalties \cite{caiso_tariff_section11_2025}). Consistent with this settlement logic, we adopt a stylized ``incentive-penalty'' mechanism that links imbalance charges to the day-ahead price, so that uninstructed deviations are not economically rewarded. 
\end{remark}

\subsection{Device Modeling}
\label{sect:device_modeling}

This subsection describes the operational models of the PV unit, energy storage unit, and aggregated load. These models define the feasible dispatch region in real-time operation.

\subsubsection{Aggregated Load and PV Unit Model}
At each hour $t\in\mathcal{T}$, the aggregated load $p_t^{\ell}$ is assumed known at the day-ahead decision time. The PV dispatch decision $p_t^{\mathrm{PV}}$ is curtailable and constrained by the available PV power
$\hat p_t^{\mathrm{PV}}$:
\begin{equation}
\label{eq:pv_bound}
0 \le p_t^{\mathrm{PV}} \le \hat p_t^{\mathrm{PV}}, \qquad \forall t\in\mathcal{T}.
\end{equation}
Here, $\hat p_t^{\mathrm{PV}}$ will be treated as an uncertain parameter to characterize the intermittency of the unrealized PV power generation and modeled in
Section~\ref{subsec:uncertainty_modeling}.

\subsubsection{Energy Storage Model} \label{sect:ES}

We model the energy storage unit using standard power and energy constraints. Let $p_t^{\ES,\mathrm{ch}}$ and $p_t^{\ES,\mathrm{dis}}$ denote the charging and discharging power at hour $t \in \mathcal{T}$, respectively. The charging and discharging powers satisfy
\begin{equation} \label{eq:es_power_bounds}
    0 \le p_t^{\mathrm{ES,ch}} \le \bar p^{\mathrm{ES}},\qquad
0 \le p_t^{\mathrm{ES,dis}} \le \bar p^{\mathrm{ES}},\qquad
p_t^{\mathrm{ES}} = -p_t^{\mathrm{ES,ch}} + p_t^{\mathrm{ES,dis}},\qquad
\forall t\in\mathcal{T}.
\end{equation}

Simultaneous charging and discharging is not allowed in practice, which is enforced through the complementarity constraint:
\begin{equation}
\label{eq:ES_cc}
    p_t^{\ES,\mathrm{ch}} \perp p_t^{\ES,\mathrm{dis}}, \quad \forall t\in\mathcal{T}.
\end{equation}
The above constraint is nonconvex, for which we introduce a tractable reformulation in the sequel.

Let $e_t$ be the state of charge (SOC) of the energy storage unit. The SOC evolves according to the charging/discharging power through
\begin{equation} \label{eq:SOC_dynamics}
    e_t = e_{t-1} + \eta^{\mathrm{ch}} p_t^{\mathrm{ES,ch}}
-  p_t^{\mathrm{ES,dis}} / \eta^{\mathrm{dis}}, \qquad \forall t\in\mathcal{T},
\end{equation}
where $\eta^{\mathrm{ch}}\in(0,1]$ and $\eta^{\mathrm{dis}}\in(0,1]$ are the charging and discharging efficiencies, respectively. For notational simplicity, the time-step length $\Delta t$ is taken as unity and is omitted throughout the remainder of the paper.

The SOC is always bounded:
\begin{equation} \label{eq:soc_bounds}
\ubar{e} \le e_t \le \bar{e}, \quad \forall t\in\mathcal{T}.
\end{equation}

In addition, we also require that the terminal SOC is equal to the initial SOC:
\begin{equation} \label{eq:terminal_condition}
    e_T = e_0.
\end{equation}

\subsubsection{Power Balance and Imbalance Power}

The real-time net injection of the VPP is the aggregated power output of the PV, energy storage, and aggregated load:
\begin{equation} \label{eq:aggregate_power}
    p_t^{\mathrm{agg}} = p_t^{\mathrm{PV}} + p_t^{\mathrm{ES}} - p_t^{\ell}.
\end{equation}

We denote the day-ahead committed net injection at hour $t\in\mathcal{T}$ by $p^\DA_t$, and define $p_t^{\mathrm{mis}}$ as the mismatch between the real-time net injection and the day-ahead commitment.
\begin{equation} \label{eq:imbalance}
     p_t^{\mathrm{mis}} = p_t^{\mathrm{agg}} - p^\DA_t.
\end{equation}

Under the convention that positive power denotes injection, $p_t^{\mathrm{mis}}<0$ indicates an injection shortfall (net import), while $p_t^{\mathrm{mis}}>0$ indicates an injection surplus (net export). The imbalance settlement cost model is discussed in Section~\ref{sect:optim_form}.

\subsection{Uncertainty Modeling}
\label{subsec:uncertainty_modeling}

Among the three device types (PV, energy storage, and aggregated load), the storage dispatch is controllable. We assume the aggregated load profile $\{p_t^\ell\}_{t\in\mathcal{T}}$ is perfectly known at the day-ahead decision time, and we explicitly model two sources of uncertainty: (i) PV availability and (ii) day-ahead prices. To capture their heterogeneous characteristics, we adopt a hybrid approach: PV availability is described by a budgeted uncertainty set to reflect its pronounced intermittency and limited predictability, while day-ahead prices are cleared by the ISO from multi-period bids/offers and thus exhibit sequential (Markovian) stochastic dynamics across time. This hybrid stochastic-robust construction combines the temporal dependence of day-ahead prices with worst-case protection against PV forecast errors.

\subsubsection{PV Availability Uncertainty via a Static Budgeted Set} \label{sect:pv_uncertainty}

We model uncertainty in PV availability (the maximum PV injection) by a budgeted uncertainty set \cite{baringo_stochastic_2018}. Let $\hat{p}^\PV := [\hat p^{\mathrm{PV}}_t]_{t\in\mathcal{T}} \in \mathbb{R}^{|\mathcal{T}|}$ denote the vector of uncertain PV availability that appears in the PV constraint in \eqref{eq:pv_bound}. We model its uncertainty set as
\begin{equation}
\label{PV uncertainty}
    \mathcal{U} = \Bigl\{ \hat{p}^\PV \in \mathbb{R}^{|\mathcal{T}|} :\; \hat{p}^\PV_t = p^{\PV,0}_t + \frac{\bar{p}_t^\PV - \ubar{p}_t^\PV}{2} \xi_t,\;\; |\xi_t| \le 1, \;\;\sum\nolimits_{t\in\mathcal{T}} |\xi_t| \le \Gamma, \quad \forall t \in \mathcal{T} \Bigr\},
\end{equation}
where $\bar{p}^\PV_t$ and $\ubar{p}^\PV_t$ are the upper and lower bounds of the PV availability forecast at $t \in \mathcal{T}$, respectively, with the nominal PV availability $p^{\PV,0}_t \coloneqq (\bar{p}^\PV_t + \ubar{p}^\PV_t)/2$. The budget parameter $\Gamma\in[0,|\mathcal{T}|]$ limits the aggregate deviation across the horizon. The nominal forecast $p^{\mathrm{PV},0}_t$ and bounds $(\bar{p}_t^\PV, \ubar{p}_t^\PV)$ can be obtained from historical data and forecasting tools (e.g., LSTM-based models). Since forecasting is not the focus of this work, the details are omitted for brevity. We assume the accuracy of the PV forecast is similar across the entire time horizon, so $\bar p_t^\PV - \ubar p_t^\PV$ is a constant for all $t \in \mathcal{T}$.

\subsubsection{Day-Ahead Price Scenarios via a Markov Process} 
\label{sect:Markov_Process}

We model day-ahead price uncertainty using a first-order, time-inhomogeneous Markov process with $T$ stages (hours) and $N$ discrete price states per stage \cite{zheng_arbitraging_2022}. Let $\mathcal{S}:=\{1,2,\ldots,N\}$ denote the set of day-ahead price states. For each hour $t\in\mathcal{T}$, we discretize historical day-ahead prices into $N$ intervals $[\ubar{\pi}_{t,s},\bar{\pi}_{t,s})$ indexed by $s\in\mathcal{S}$, and represent state $s$ by a representative price $\pi_{t,s}$. Thus, the day-ahead clearing price at hour $t$ is modeled as taking the value $\pi_{t,s}$ when the process is in state $s$.

Let $S_t \in\mathcal{S}$ denote the day-ahead price state at hour $t$ and $\vartheta_{s,s',t} := \mathbb{P}(S_{t+1}=s' \mid S_t=s)$ denote the transition probability from state $s$ at hour $t$ to state $s'$ at hour $t{+}1$. Let $\mathcal{D}$ denote the set of historical days used for training. Using historical day-ahead price data $\{\lambda^{\DA}_{t,d}\}_{t\in\mathcal{T},d\in\mathcal{D}}$, we estimate these probabilities by counting the observed state transitions in the historical data:
\begin{equation}
\label{eq:hourly_transition_probabilities}
    \vartheta_{s,s',t} = \frac{ \sum_{d\in\mathcal{D}} \mathbb{I}\!\left\{\lambda^{\DA}_{t,d}\in[\ubar{\pi}_{t,s},\bar{\pi}_{t,s})\right\} \mathbb{I}\!\left\{\lambda^{\DA}_{t+1,d}\in[\ubar{\pi}_{t+1,s'},\bar{\pi}_{t+1,s'})\right\}
    }{
    \sum_{d\in\mathcal{D}} \mathbb{I}\!\left\{\lambda^{\DA}_{t,d}\in[\ubar{\pi}_{t,s},\bar{\pi}_{t,s})\right\} }, \qquad \forall s,s'\in\mathcal{S},\ t\in\mathcal{T}\setminus \{T\},
\end{equation}
where $\mathbb{I}\{\cdot\}$ is the indicator function.
Monte Carlo sampling of this Markov model generates a finite set of price trajectories $\Omega=\{1,2,\ldots,W\}$. Each trajectory $\omega\in\Omega$ corresponds to a sampled state sequence $\{s_t(\omega)\}_{t\in\mathcal{T}}$ and a price sequence
\begin{equation}
\lambda^{\DA}_{t,\omega} := \pi_{t,s_t(\omega)}, \qquad \forall t\in\mathcal{T},\ \omega\in\Omega.
\end{equation}
These day-ahead price scenarios enter the first-stage settlement term in the 2S-ARO formulation. Real-time prices are not modeled through an additional stochastic process; instead, follow the market rules in Remark \ref{remark real time price}, real-time settlements are captured by the incentive-penalty imbalance mechanism linked to the realized day-ahead clearing price, as detailed in the sequel.

\subsection{Optimization Formulation} \label{sect:optim_form}

We now formulate the VPP day-ahead offering problem as a two-stage stochastic ARO model, where the day-ahead price uncertainty is represented by scenarios $\omega \in \Omega = \{1, 2, \ldots, W\}$ generated from the Markov model in Section \ref{sect:Markov_Process}. For each scenario $\omega$, the Markov model gives the scenario probability $\rho_\omega$ and the realized day-ahead price trajectory $\{\lambda^\DA_{t,\omega}\}_{t\in\mathcal{T}}$.

\paragraph{First-stage decision variables: the day-ahead offer quantities}

The VPP submits a stepwise day-ahead offer curve for each $t\in\mathcal{T}$, specified by the offered quantities at price states $s\in\mathcal{S}$. Let 
\begin{equation}
    x := p^\DA = \bigl\{ p^\DA_{t,s} \bigr\}_{t\in\mathcal{T},s\in\mathcal{S}}
\end{equation}
be the first-stage decision vector. Each component $p^\DA_{t,s}$ is the offered net injection at
hour $t$ if the DA clearing price falls in state $s$. Nonanticipativity is implicitly enforced by this $(t,s)$ indexing: all scenarios that realize the same price at hour $t$ share the same offered quantity $p^\DA_{t,s}$. Note that the dimension of $x$ is $|\mathcal{S}||\mathcal{T}|$, which is independent of the number of sampled scenarios. The feasible set $x\in\mathcal{X}$ enforces operational limits and the monotonicity of the offer curve (i.e., higher prices correspond to higher offered quantities):
\begin{equation} \label{first_stage_FR}
    \mathcal{X} = \bigl\{ p^\DA \in \mathbb{R}^{|\mathcal{S}||\mathcal{T}|} :  
    \ubar{p}^{\DA} \leq p_{t,s}^{\DA} \leq \bar{p}^{\DA},  \quad p_{t,s}^{\DA} \leq p_{t,s'}^{\DA} \text{ if } \pi_{t,s} \leq \pi_{t,s'}, \quad \forall t \in \mathcal{T}, s,s'\in\mathcal{S} \bigr\}, 
\end{equation}
where $\ubar{p}^{\DA}$ and $\bar{p}^{\DA}$ denote the minimum and maximum offering quantities, respectively.

\paragraph{Second-stage decision variables: the real-time dispatch signals}

After the day-ahead market clears and the PV availability uncertainty is realized, the VPP dispatches its DERs in real time. PV availability is modeled by the budgeted uncertainty set $\mathcal{U}$ in Section \ref{sect:pv_uncertainty}. Given a realization $\hat{p}^{\mathrm{PV}}\in\mathcal{U}$, the second-stage dispatch variables include
\begin{equation}
    y := \bigl\{ p_t^{\mathrm{PV}},\, p_t^{\mathrm{ES,ch}},\, p_t^{\mathrm{ES,dis}},\, e_t,\, p_t^{\mathrm{mis}} \bigr\}_{t\in\mathcal{T}}.
\end{equation}

For a fixed day-ahead schedule and PV availability realization, the feasible region of the real-time DER dispatch is
\begin{equation} \label{eq:Y_DER}
    \mathcal{Y}\bigl(p^{\DA},\hat{p}^{\mathrm{PV}}\bigr) := \mathcal{Y}^{\mathrm{PV}}\bigl(\hat{p}^{\mathrm{PV}}\bigr) \;\cap\; \mathcal{Y}^{\mathrm{ES}} \;\cap\; \mathcal{Y}^{\mathrm{mis}}\bigl(p^{\DA}\bigr),
\end{equation}
where $\mathcal{Y}^{\PV}(\hat{p}^{\PV})$ denotes the set of $y$ feasible to the PV feasibility constraint \eqref{eq:pv_bound}. Similarly, $\mathcal{Y}^{\ES}$ encodes the energy storage constraints \eqref{eq:es_power_bounds}--\eqref{eq:terminal_condition}, and $\mathcal{Y}^{\mathrm{mis}}(p^{\DA})$ enforces the power balance incorporating the mismatch between real-time DER dispatch and day-ahead commitment \eqref{eq:aggregate_power}--\eqref{eq:imbalance}.

\paragraph{Real-time cost model}
To comply with U.S. market rules (see Remark~\ref{remark real time price} in Section~\ref{subsec:decision_model}), we introduce the incentive-penalty imbalance settlement mechanism intended to discourage arbitrage between day-ahead and real-time settlements. Recall that
$p_t^{\mathrm{mis}} = p_t^{\mathrm{agg}} - p_t^\DA$, so $p_t^{\mathrm{mis}}<0$ indicates an injection
shortfall (net import) and $p_t^{\mathrm{mis}}>0$ indicates an injection surplus (net export). Under scenario $\omega$, we set two settlement rates linked to the realized day-ahead clearing price $\lambda^\DA_{t,\omega}$:
\begin{equation} \label{eq:imb_rates}
    \lambda^{\mathrm{imp}}_{t,\omega} := \lambda^\DA_{t,\omega} + \kappa, \qquad
    \lambda^{\mathrm{exp}}_{t,\omega} := \lambda^\DA_{t,\omega} - \kappa, \qquad 
    \forall t\in\mathcal{T},
\end{equation}
where $\kappa > 0$ is a reduced-form parameter representing the imbalance charges induced by the settlement rules in Remark \ref{remark real time price}. Thus, importing energy to cover a shortfall is priced above the day-ahead clearing price, while exporting surplus energy is credited below the day-ahead clearing price. The resulting imbalance settlement cost is
\begin{equation} \label{eq:imb_cost_kappa}
    S\bigl(p_t^{\mathrm{mis}};\lambda^\DA_{t,\omega}\bigr)
    =
    \lambda^{\mathrm{imp}}_{t,\omega}\,[-p_t^{\mathrm{mis}}]_+ -
    \lambda^{\mathrm{exp}}_{t,\omega}\,[p_t^{\mathrm{mis}}]_+,
\end{equation}
where $[x]_+ := \max\{x,0\}$. For brevity, all terms involving $\Delta t$ that multiply the power variables in the objective function are omitted. The cost is positive under net import and negative under net export. This incentive-penalty imbalance settlement mechanism captures the fact that real-time deviations are settled at imbalance charges different from the day-ahead price and may face additional deviation penalties, and it guarantees that intentional deviations are not rewarded in the price-taking setting. As noted in \cite{kim_benefits_2021}, if deviation prices do not respond to day-ahead prices and the VPP has unrestricted flexibilities for the real-time operation, the offering problem can degenerate, and stochastic offers may provide no advantage over deterministic ones; \eqref{eq:imb_cost_kappa} mitigates this behavior.

For a given day-ahead price scenario $\omega$, the real-time cost is defined as the sum of the DER operating cost and the imbalance settlement cost over the time horizon $\mathcal{T}$:
\begin{equation} \label{eq:Pi_RT}
    \Phi^{\RT}(y;\omega)
    :=
    \sum_{t\in\mathcal{T}}
    \Bigl(
        c^\PV p_t^\PV + c^\ES p_t^{\ES,\mathrm{dis}} + S(p_t^{\mathrm{mis}};\lambda^\DA_{t,\omega})
    \Bigr),
\end{equation}
where $c^\PV$ denotes the marginal cost of PV and $c^\ES$ accounts for the battery degradation cost \cite{qin_role_2023}. The PV operating cost model follows the variable operations and maintenance energy (VOM-EN) pricing model \cite{caiso_tariff_section11_2025}.

\paragraph{The two-stage stochastic ARO problem}

For each price scenario $\omega\in\Omega$, the VPP earns day-ahead settlement revenue $\sum_{t\in\mathcal{T}} \lambda^\DA_{t,\omega} p^\DA_{t,s_t(\omega)}$ and incurs real-time operating and imbalance costs that depend on the dispatch decisions and PV availability. Since PV availability errors reflect forecast uncertainty that we wish to hedge against, we optimize against its worst-case realization in $\mathcal{U}$. On the other hand, since day-ahead prices exhibit stochastic variability with learnable temporal structure, we represent them by a scenario set $\Omega$ and optimize the expected profit. Given day-ahead price $p^\DA$, the net cost (the difference between real-time cost and day-ahead revenue) can be represented as
\begin{equation}
\Phi(p^{\mathrm{DA}})
\;:=\; \sum_{\omega \in \Omega} \rho_\omega \Bigl( -\underbrace{\sum_{t\in\mathcal{T}}\lambda^{\DA}_{t,\omega}\, p^{\DA}_{t,s_t(\omega)}}_{\text{day-ahead settlement}} + \underbrace{\max_{\hat{p}^{\mathrm{PV}}\in\mathcal{U}} \;\min_{y\in\mathcal{Y}(p^{\DA},\hat{p}^{\mathrm{PV}})} \; \Phi^{\RT}(y;\omega)}_{\text{worst-case real-time cost}} \Bigr)
\end{equation}
As mentioned, the first-stage decision $p^\DA$ (day-ahead offer) does not scale with the number of samples $|\Omega|$, while the second-stage decision $y$ does. The overall VPP offering problem is
\begin{equation}
\label{eq:2S-ARO}
    \min_{p^{\mathrm{DA}}\in\mathcal{X}} \;
\Phi(p^{\mathrm{DA}}).
\end{equation}

%% file: reformulation_cost-to-go.tex
\section{Reformulation of the Optimization Model}
\label{reformulated formulation}

In this section, we present the reformulation of the second-stage problem of the two-stage stochastic ARO problem in \eqref{eq:2S-ARO} and then decouple the max-min structure of the second-stage problem, which facilitates our subsequent development of efficient solution algorithms.

\subsection{Reformulation of the Second-Stage Problem}
\label{Reformulated cost-to-go}
For each price scenario $\omega\in \Omega$, the scenario-specific value function is defined as
\begin{equation}
    V_\omega(p^\DA) \coloneqq \max_{\hat{p}^{\mathrm{PV}}\in\mathcal{U}} \;\min_{y\in\mathcal{Y}(p^{\DA},\hat{p}^{\mathrm{PV}})} \; \Phi^{\RT}(y;\omega),
\end{equation}
which represents the worst-case optimal cost of the second-stage problem given the first-stage decision $p^{\mathrm{DA}}$. Then the 2S-ARO model \eqref{eq:2S-ARO} can be rewritten in the following form:
\begin{equation}
\label{eq: nested 2s-aro}
\min_{p^{\mathrm{DA}}\in \mathcal{X}} \sum_{\omega \in \Omega} \rho_\omega \Bigl( -\sum_{t\in\mathcal{T}}\lambda^{\DA}_{t,\omega}\, p^{\DA}_{t,s_t(\omega)} + V_\omega(p^\DA) \Bigr).
\end{equation}
In what follows, we fix a scenario $\omega \in \Omega$ and study the structure of $V_\omega(p^\DA)$. Given $p^{\mathrm{DA}}\in \mathcal{X}$, we can rewrite the 
formulation of $V_\omega(\cdot)$ in \eqref{eq:2S-ARO} in the following form 
\begin{equation}
\label{eq:value-to-go function}
V_\omega(p^{\mathrm{DA}} ) =\max_{\hat{p}^{\mathrm{PV}}\in \mathcal{U}} \,\,V_\omega( \hat{p}^{\mathrm{PV}},p^{\mathrm{DA}} ) ,
\end{equation}
where $V_\omega( \hat{p}^{\mathrm{PV}},p^{\mathrm{DA}} )$ is the so-called cost-to-go function, given as 
\begin{equation}
\label{eq:cost-to-go function}
V_\omega( \hat{p}^{\mathrm{PV}},p^{\mathrm{DA}} ) =\min_{\,y\in \mathcal{Y} (p^{\mathrm{DA}},\hat{p}^{\mathrm{PV}})} \,\,\Phi ^{\mathrm{RT}}(y; \omega).
\end{equation}

To facilitate the discussion, let us rewrite the imbalance settlement cost in \eqref{eq:imb_cost_kappa} as:
\begin{equation}\label{eq:RT_cost}
\begin{aligned}
S\bigl( p_{t}^{\mathrm{mis}};\lambda _{t,\omega}^{\mathrm{DA}} \bigr) &= \max \,\,\bigl\{ -\lambda _{t,\omega}^{\mathrm{imp}}p_{t}^{\mathrm{mis}}, -\lambda _{t,\omega}^{\exp}p_{t}^{\mathrm{mis}} \bigr\} 
\\ 
&=\max \,\,\bigl\{ \lambda _{t,\omega}^{\mathrm{imp}}\left( p_{t}^{\mathrm{DA}}-p_{t}^{\mathrm{PV}}+p_{t}^{\ell}-p_{t}^{\mathrm{ES}} \right) , \lambda _{t,\omega}^{\exp}\bigl( p_{t}^{\mathrm{DA}}-p_{t}^{\mathrm{PV}}+p_{t}^{\ell}-p_{t}^{\mathrm{ES}} \bigr) \bigr\} .
\end{aligned}
\end{equation}

To further reformulate the imbalance settlement cost, we show that the ``no simultaneous charging and discharging'' constraint \eqref{eq:ES_cc} is redundant and can be dropped without affecting the optimality of problem \eqref{eq:cost-to-go function}:
\begin{proposition}
\label{NoCDischarge}
Given a price scenario $\omega\in\Omega$ and a PV realization $\hat p^{\mathrm{PV}}\in\mathcal U$.
As long as i) the imbalance settlement rates in \eqref{eq:imb_rates} satisfy
$\lambda^{\mathrm{imp}}_{t,\omega}, \lambda^{\mathrm{exp}}_{t,\omega} > 0$ for all $t\in\mathcal T$ (that is, $\lambda^{\mathrm{DA}}_{t,\omega}>\kappa$); ii) $\eta^{\mathrm{ch}}, \eta^{\mathrm{dis}} \in (0,1]$; and iii) $c^{\mathrm{ES}}\ge 0$, then we have 
\begin{equation}
\bigl(p_{t}^{\mathrm{ES,ch}}\bigr)^*\bigl(p_{t}^{\mathrm{ES,dis}}\bigr)^*=0,\qquad \forall t\in\mathcal T
\end{equation}
for the optimal solution to the second-stage problem \eqref{eq:cost-to-go function} without constraint \eqref{eq:ES_cc}.
\end{proposition}
The proof of Proposition \ref{NoCDischarge} can be found in \ref{sect:app:NSCD}.

According to Proposition \ref{NoCDischarge}, the complementarity constraint \eqref{eq:ES_cc} of energy storage can be omitted. At optimality, the storage is either charging or discharging at each time $t$, but never both simultaneously. Hence, the imbalance settlement cost \eqref{eq:RT_cost} can be further reformulated as: 
\begin{align} \label{eq:RT_cost1}
\max \bigl\{ &
\lambda_{t,\omega}^{\mathrm{imp}}
( p_{t}^{\mathrm{DA}} - p_{t}^{\mathrm{PV}} + p_{t}^{\ell} - p_{t}^{\mathrm{ES,dis}} ),
\lambda_{t,\omega}^{\mathrm{imp}}
( p_{t}^{\mathrm{DA}} - p_{t}^{\mathrm{PV}} + p_{t}^{\ell} + p_{t}^{\mathrm{ES,ch}} ), \notag\\
& \lambda_{t,\omega}^{\exp}
( p_{t}^{\mathrm{DA}} - p_{t}^{\mathrm{PV}} + p_{t}^{\ell} - p_{t}^{\mathrm{ES,dis}} ),
\lambda_{t,\omega}^{\exp}
( p_{t}^{\mathrm{DA}} - p_{t}^{\mathrm{PV}} + p_{t}^{\ell} + p_{t}^{\mathrm{ES,ch}} )
\bigr\}.
\end{align}

If we define an auxiliary variable $\Delta e_{t}\coloneqq \eta ^{\mathrm{ch}}p_{t}^{\mathrm{ES},\mathrm{ch}}-p_{t}^{\mathrm{ES},\mathrm{dis}}/\eta^{\mathrm{dis}}$, it follows from \eqref{eq:SOC_dynamics} that $\Delta e_t = e_t - e_{t-1}$. The energy rating constraint \eqref{eq:soc_bounds} and the terminal SOC constraint \eqref{eq:terminal_condition} can then be equivalently reformulated as: 
\begin{equation} \label{eq:ES_bounds}
    \ubar{e} - e_0 \le \sum_{\tau =1}^t{\Delta e_\tau}\le \bar{e} - e_0, \quad \forall t\in \mathcal{T} \setminus \{T\}, \qquad \sum_{\tau =1}^T{\Delta e_\tau}=0,
\end{equation}
which we denote by $\mathcal{Y}_e \subseteq \mathbb{R}^T$.
The real-time cost $\Phi^\RT (y;\omega)$ at time $t$ in \eqref{eq:Pi_RT}, which is the sum of the imbalance settlement cost \eqref{eq:RT_cost1}, the energy storage operating cost $c^{\mathrm{ES}} p_{t}^{\mathrm{ES,dis}}$, and the PV operating cost $c^\PV p_t^\PV$, can be reformulated as
\begin{align} \label{eq:imbalance cost reformulated}
    f_{t,\omega}(\Delta e_t, p_t^\PV, p_t^\DA) \coloneqq \max \bigl\{ & \tilde c_{t,\omega}^\mathrm{dis} \Delta e_t + \utilde c_{t,\omega}^\PV p_t^\PV + \lambda_{t,\omega}^\mathrm{imp} (p_t^\DA + p_t^\ell), \notag \\
    & \tilde c_{t,\omega}^\mathrm{ch} \Delta e_t + \utilde c_{t,\omega}^\PV p_t^\PV + \lambda_{t,\omega}^\mathrm{imp} (p_t^\DA + p_t^\ell), \notag \\
    & \utilde c_{t,\omega}^\mathrm{dis} \Delta e_t + \tilde c_{t,\omega}^\PV p_t^\PV + \lambda_{t,\omega}^\mathrm{exp} (p_t^\DA + p_t^\ell), \notag \\
    & \utilde c_{t,\omega}^\mathrm{ch} \Delta e_t + \tilde c_{t,\omega}^\PV p_t^\PV + \lambda_{t,\omega}^\mathrm{exp} (p_t^\DA + p_t^\ell),
    \bigr\}
\end{align}
where $\tilde c_{t,\omega}^\mathrm{dis} \coloneqq \eta^\mathrm{dis}(\lambda^\DA_{t,\omega} + \kappa - c^\ES)$, $\tilde c_{t,\omega}^\mathrm{ch} \coloneqq (\lambda^\DA_{t,\omega} + \kappa) / \eta^{\mathrm{ch}}$, $\utilde c_{t,\omega}^\mathrm{dis} \coloneqq \eta^\mathrm{dis}(\lambda^\DA_{t,\omega} - \kappa - c^\ES)$, $\utilde c_{t,\omega}^\mathrm{ch} \coloneqq (\lambda^\DA_{t,\omega} - \kappa) / \eta^{\mathrm{ch}}$, are the effective charging/discharging prices, and $\tilde c_{t,\omega}^\PV \coloneqq c^\PV - \lambda^\DA_{t,\omega} + \kappa$, $\utilde c_{t,\omega}^\PV \coloneqq c^\PV - \lambda^\DA_{t,\omega} - \kappa$ are the effective PV prices.

After all above derivation, the cost-to-go function $V_\omega(\hat{p}^{\mathrm{PV}},p^{\mathrm{DA}})$ can be reformulated as the following problem with a piecewise-linear objective function:
\begin{subequations}
\label{eq:reformulated_cost_to_go}
\begin{align}
V_\omega(\hat p^\PV, p^\DA) = \min_{\Delta e, p^\PV} \quad
&\sum_{t\in \mathcal{T}}
f_{t,\omega}(\Delta e_t, p_t^\PV, p_t^\DA)
\\
\text{s.t.}\quad 
& 0 \le p_t^{\mathrm{PV}} \le \hat p_t^{\mathrm{PV}}, \qquad \forall t\in\mathcal{T} \label{eq:reformulated_cost_to_go_pv}\\
& \Delta e \in \mathcal{Y}_e. \label{eq:reformulated_cost_to_go_es}
\end{align}
\end{subequations}

\subsection{Decoupling of the Max-Min Problem}

The second-stage problem (the value function in \eqref{eq:value-to-go function}) is a non-convex max-min coupled problem, which is computationally challenging in its original form. As mentioned in \cite{sun_adaptive_2014}, for the piecewise-linear value function of the 2S-ARO problem, evaluating the value function \eqref{eq:value-to-go function} for any fixed first-stage decision variable ($p^\DA$) is hard. As discussed in Section~\ref{Literature Review}, the standard approaches include transforming this problem into an MILP problem, whose computational performance is limited in industrial applications. For the second-stage problem in this day-ahead offering problem, in which the PV generation follows budgeted uncertainty set, a more computationally efficient recourse counterpart is developed in this section.

\begin{proposition}[Decoupling of the second-stage problem]
\label{DecouplingCons}
Let a price scenario $\omega\in\Omega$ and a PV generation budgeted uncertainty set $\mathcal U$ be given. We assume that day-ahead prices are positive and higher than the PV marginal cost, i.e., $\lambda_{t,\omega}^{\DA} > c^{\PV}  > 0,  \forall t\in\mathcal{T},\ \forall\omega\in\Omega$.
As long as 
\begin{equation}
\label{eq:kappa_cond_prop}
0 < \kappa \le 
\min\nolimits_{\omega\in\Omega, t, t'\in\mathcal{T}, t \neq t'} \Bigl\{
\bigl|\lambda_{t,\omega}^{\DA}-\lambda_{t',\omega}^{\DA}\bigr|/2,
\ \lambda_{t,\omega}^{\DA}-c^{\PV}
\Bigr\},
\end{equation}
we have
\begin{enumerate}[label=\roman*), leftmargin=*]
    \item the optimal dispatch decision for PV generation in the second-stage problem satisfies
    \begin{equation}
    p_t^{\PV,*}(\omega)=\hat p_t^{\PV,*}(\omega),\quad \forall t\in\mathcal T,
    \end{equation}
    where $\hat p^{\PV,*}(\omega)$ can be chosen as an optimal solution of
    \begin{equation}
    \label{eq:worst_case_pv_prop}
    \max_{\hat p^{\PV}\in\mathcal U}
    \sum_{t\in\mathcal T}-\lambda_{t,\omega}^{\DA} \hat p_t^{\PV};
    \end{equation}
    \item Given $p_t^{\PV,*}(\omega)$, the second-stage problem \eqref{eq:value-to-go function} reduces to the deterministic problem
    \begin{equation}\label{eq:recourse_given_worst_pv}
    \ubar V_\omega(p^{\DA})
    \coloneqq\min_{\Delta e \in \mathcal{Y}_e}
    \sum_{t\in\mathcal{T}}
    f_{t,\omega}(\Delta e_t, p_t^{\PV,*}, p_t^\DA)
    \end{equation}
\end{enumerate}
\end{proposition}
The proof of Proposition \ref{DecouplingCons} can be found in \ref{sect:app:DPC}.

Proposition~\ref{DecouplingCons} implies that, for any fixed $p^{\DA}$ and $\omega$, the original second-stage problem can be evaluated in a sequential manner. The worst-case PV availability profile is first obtained from \eqref{eq:worst_case_pv_prop}; after fixing $p^{\PV,*}(\omega)$, the second-stage problem reduces to \eqref{eq:recourse_given_worst_pv}. Therefore, the original VPP day-ahead offering problem in \eqref{eq:2S-ARO} can in fact be reformulated as a two-stage stochastic LP problem with a convex piecewise-linear objective function and a polyhedral feasible region, as shown below:
\begin{equation}
\label{eq:2S-LP}
\min_{p^{\mathrm{DA}} \in \mathcal{X}} \;
\sum_{\omega \in \Omega} \rho_{\omega} 
\Bigl(
    -\sum_{t \in \mathcal{T}}
        \lambda_{t,\omega}^{\mathrm{DA}} \,
        p_{t,s_t(\omega)}^{\mathrm{DA}}
    + 
    \min_{\Delta e \in \mathcal{Y}_e}
    \sum_{t \in \mathcal{T}} f_{t,\omega}(\Delta e_t, p_t^{\PV,*}(\omega), p_t^\DA)
\Bigr) 
\end{equation}
where $p_{t}^{\mathrm{PV},*}(\omega) \in 
\arg\max_{\hat{p}^{\mathrm{PV}}\in \mathcal{U}}
\sum_{t\in \mathcal{T}} -\lambda_{t,\omega}^{\DA} \hat{p}_{t}^{\mathrm{PV}}$ can be optimized independently without coupling the second-stage optimization in $f_{t,\omega}\left( \Delta e_t, p_t^\PV, p_t^\DA \right)$. This decoupling replaces the original second-stage max-min evaluation in \eqref{eq:value-to-go function} with a linear maximization over $\mathcal U$ and a deterministic minimization of piecewise-linear objective over linear constraints, thereby avoiding complicated MILP reformulations. The final two-stage stochastic LP reformulation in \eqref{eq:2S-LP} of the offering problem looks quite different from the functional form \eqref{eq:2S-ARO}. However, the main advantage of the reformulation in \eqref{eq:2S-LP} is that it reveals the separable convex piecewise-linear structure of the recourse problem, which is essential for the efficient algorithm development in the sequel.

%  z_{t,\omega}'(p^{\PV,*}

%% file: LSA.tex
\subsection{An Exact Greedy Oracle for the Decoupled Second-Stage Problem}
\label{Greedy Algorithm}

Following the reformulated second-stage problem in \eqref{eq:2S-LP}, the original non-convex second-stage problem exhibits decoupled structure between $\hat p^{\PV}$ and $\Delta e$, so that after fixing the worst-case PV realization $p_t^{\PV,*}(\omega)$, the second-stage problem reduces to the following deterministic problem based on Proposition \ref{DecouplingCons}:
\begin{equation} \label{eq:2ndcompact}
    \min_{\Delta e \in \mathcal{Y}_e}
    \sum_{t\in\mathcal{T}}
    f_{t,\omega}(\Delta e_t, p_t^{\PV,*}, p_t^\DA),
\end{equation}
where $\mathcal{Y}^e$ is the feasible set induced by the storage cumulative sum constraints and the terminal SOC requirement, and each $f_{t,\omega}(\cdot)$ is a convex piecewise-linear function of the scalar variable $\Delta e_t$. Hence, for each scenario $\omega$, the recourse evaluation is a separable convex piecewise-linear optimization problem over nested cumulative sum constraints.

Problems of this form are classical in resource allocation settings. Accordingly, our goal in this subsection is not to claim a new general-purpose algorithm for this problem class. Instead, we exploit the specific structure revealed by the reformulation to build a lightweight \emph{exact recourse oracle} that is particularly convenient for repeated scenario-wise evaluations inside the projected subgradient framework to be developed in Section \ref{Projected Subgradient Method}. In addition to the optimal value of \eqref{eq:2ndcompact}, the oracle also records the active affine piece of each $f_{t,\omega}$ at the optimum, which will later be used to assemble a subgradient with respect to the first-stage decision $p^{DA}$.

The specialized greedy algorithm to solve Problem \eqref{eq:2ndcompact} is presented in Algorithm \ref{alg:local_search_algorithm}. The key observation behind the algorithm is that decreasing a single coordinate $\Delta e_t$ by $\delta$ decreases all subsequent SOC variables $e_{t'}, t' > t$, by the same amount. Thus, although the objective is separable in $\Delta e$, the feasibility constraints couple the variables through a nested cumulative sum structure. The greedy algorithm exploits this coupling constraint together with the piecewise-linear structure of the objective function.

Rather than starting from an arbitrary feasible point, Algorithm \ref{alg:local_search_algorithm} starts from the point that upper bounds every feasible solution componentwise:
\begin{equation}
    \Delta e_t^{(0)} = \bar e_t - \ubar e_{t-1}, \quad t\in\mathcal{T}.
\end{equation}
The choice is convenient for two reasons. First, all subsequent updates are monotone: the algorithm only decreases the coordinates of $\Delta e$. Second, every update affects only the suffix of the SOC trajectory, which makes it easier to maintain feasibility and detect binding constraints.

For each time step $t$, the scalar function $f_{t,\omega}(\Delta e_t; \cdot)$ has four affine pieces. We rank the slopes of all these pieces across $T$ time steps in a lexicographically decreasing ranked list $\mathcal{A}$. At each iteration, let $\hat {\alpha}_{\tau,\iota}$ be the first entry of $\mathcal{A}$, which indicates the time step to be considered for update. If the time step admits an improving update, the algorithm performs a descent update to the solution $\Delta e$ by decreasing $\Delta e_\tau$. The admissible step is chosen as the largest decrement that remains on the current affine piece and preserves the nested SOC structure. After each update, the algorithm removes from $\mathcal{A}$ any affine pieces or time steps that can no longer become active under subsequent monotone decreases. In particular, if the current affine piece is exhausted, it is deleted immediately. If a cumulative sum constraint becomes binding in a way that blocks further updates on some time steps, then the corresponding active pieces are removed as well. Since $\mathcal{A}$ is finite, this removal mechanism ensures finite termination and avoids redundancy.

For notation convenience, we define that, for each iterate $k$ and time $\tau \in \mathcal{T}$, $e_{\tau}^{(k)} \coloneqq e_0 + \sum_{j=1}^{\tau}\Delta e_j^{(k)}$; this is also the definition of the SOC at that time. For each $\tau \in \mathcal{T}$, we define the time-dependent SOC upper bound $\bar{e}_{\tau}$ and lower bound $\ubar{e}_{\tau}$ as: $\bar{e}_{\tau}=\bar e,\ \ubar{e}_{\tau}=\ubar e, \ \forall \tau \in \mathcal{T}\setminus\{T\}$, $\bar{e}_0 =\ubar{e}_0= e_0$ , and $\bar{e}_{T}=\ubar{e}_{T}=e_0$. We denote by $a_t(\Delta e_t)$ the left derivative, $a_t(\Delta e_t) := \lim_{\epsilon \to 0^-} \mathrm{d}f_{t,\omega}(\Delta e_t+\epsilon;\cdot)/\mathrm{d}\Delta e_t$. Let $\{\hat\alpha_{t,i}\}_{i=1}^{4}$ be the slopes of the affine pieces of the piecewise linear function $f_{t,\omega}(\Delta e_t;\cdot)$ in \eqref{eq:2S-LP} and $\mathcal A:=\{\hat\alpha_{t,i}\}_{t\in[\mathcal T],\ i=[4]}$ be the lexicographically decreasing list obtained by sorting the slopes of all pieces of $f_{t,\omega}(\Delta e_t;\cdot)$. Since the algorithm decreases $\Delta e_t$, a larger left derivative $a_t(\Delta e_t)$ corresponds to a larger first-order reduction in the objective per unit decrement.

\begin{algorithm}[!t]
\caption{Greedy Oracle for the Second-Stage LP Problem}
\label{alg:local_search_algorithm}
\begin{algorithmic}[1]
\REQUIRE Price scenario $\omega$, first-stage decision $p^\DA$, PV portfolio $p^{\PV,*}(\omega)$, horizon $\mathcal{T}$; SOC bounds $e_0,\bar e, \ubar e$; time-dependent SOC bounds $\bar{e}_{\tau}, \ubar{e}_{\tau}, \forall \tau \in \mathcal{T}$; list set $\mathcal A^{(0)}$.
\ENSURE An optimal solution $\Delta e^*$ of the second-stage LP problem in \eqref{eq:2ndcompact}.
\STATE \textbf{Initialization:} $\Delta e_t^{(0)} \gets \bar{e}_t - \ubar{e}_{t-1}, t \in \mathcal{T}$; set $k\gets 0$
\WHILE{$\mathcal \mathcal{A}^{(k)}\neq\varnothing$} 
\STATE Let $\hat\alpha_{t,\iota}$ be the first element of $\mathcal A$ 
    \IF{$a_t(\Delta e_t^{(k)})<\hat\alpha_{t,\iota}$}
        \STATE $\Delta e_{\tau}^{(k+1)} \gets \Delta e_{\tau}^{(k)}, \forall \tau \in \mathcal {T}$; $\mathcal{A}^{(k+1)} \gets \mathcal{A}^{(k)} \setminus \{\hat\alpha_{t,\iota}\}$ \label{line:na}
    \ELSIF{$a_t(\Delta e_t^{(k)}) \le 0$}
        \IF{$\Delta e^{(k)} \in \mathcal{Y}_e$}
            \STATE \textbf{return} $\Delta e^{(k)}$ \label{line:early_return}
        \ELSE
            % \STATE Let $\tau$ be the first time step such that $e_t^{(k)} - \bar e_t > 0$.  Update $\Delta e_\tau^{(k+1)} \gets \Delta e_\tau^{(k)} - (e_t^{(k)} - \bar e_t)$
            \STATE $\tau \gets \min \{\tau: e_\tau^{(k)} - \bar e_\tau > 0\}$; let $\hat \alpha_{\tau', i}$ be the first element in $\mathcal{A}$ such that $\tau' \le \tau$.
            \STATE $\delta_1 \gets \sup \{ \delta: a_{\tau'}(\Delta e_{\tau'}^{(k)} - \delta) \ge \hat{\alpha}_{\tau',i} \}$; $\;\;\delta_2 \gets \min_{\tau' \le s < \tau} \{e_s^{(k)} - \ubar e_s\}$; $\;\;\delta_3 \gets e_\tau^{(k)} - \bar e_\tau$ \label{line:d1d2d3}
            \STATE $\delta^* \gets \min\{\delta_1, \delta_2, \delta_3\}$;  $\Delta e_{\tau'}^{(k+1)} \gets \Delta e_{\tau'}^{(k)} - \delta^*,\ \Delta e_s^{(k+1)} \gets \Delta e_s^{(k)}, \ \forall s \in \mathcal{T}\setminus\{\tau'\}$ \label{line:update_negslope}
            \IF{$\delta^*=\delta_1$}
                \STATE $\mathcal{A}^{(k+1)} \gets \mathcal{A}^{(k)} \setminus \{\hat \alpha_{\tau', i}\}$ \label{line:d1remove1}
            \ELSIF{$\delta^*=\delta_2$}
                \STATE $\tau^* \gets \max_{\tau' \le s < \tau} \{s: e_{s}^{(k)} - \ubar{e}_{s} = \delta_2\}$; $\;\; \mathcal{A}^{(k+1)} \gets \mathcal{A}^{(k)} \setminus \{\hat\alpha_{s,i}\}_{s \in [\tau^*],\, i \in [4]}$ \label{line:d2remove1}
            \ENDIF
        \ENDIF
    \ELSE
        \STATE $\delta_1 \gets \sup \{ \delta: a_t(\Delta e_t^{(k)} - \delta) \ge \hat{\alpha}_{t,\iota} \}$; $\delta_2 \gets \min_{t \le \tau \le T}\{e_{\tau}^{(k)} - \ubar{e}_{\tau}\} - \max_{1 \le \tau \le t-1}\{[e_{\tau}^{(k)} - \bar{e}_{\tau}]_+\}$  \label{line:d1d2}
        \STATE $\delta^* \gets \min \{\delta_1,\delta_2\}$; Update
        $\Delta e_t^{(k+1)} \gets \Delta e_t^{(k)} - \delta^*,\
        \Delta e_\tau^{(k+1)} \gets \Delta e_\tau^{(k)},
        \ \forall \tau \in \mathcal{T}\setminus\{t\}$ \label{line:update}
        \IF{$\delta^* = \delta_1$}
            \STATE $\mathcal{A}^{(k+1)} \gets \mathcal{A}^{(k)} \setminus \{\hat\alpha_{t,\iota}\}$ \label{line:d1remove2}
        \ELSIF{$\delta^* = \delta_2$}
            \IF{$\min_{t \le \tau \le T}\{e_{\tau}^{(k)} - \ubar{e}_{\tau}\} = \delta_2$}
                \STATE $\tau \gets \max_{t\le s \le T} \{s : e_s^{(k)} - \ubar{e}_s = \delta_2\}$; $\;\;\mathcal{A}^{(k+1)} \gets \mathcal{A}^{(k)} \setminus \{\hat\alpha_{s,i}\}_{s \in [\tau],\, i \in [4]}$ \label{line:d2remove2}
            \ELSE
                \STATE $\tilde \delta \gets \max_{1 \le \tau \le t-1}\{e_{\tau}^{(k)} - \bar{e}_{\tau}\}$; $\utilde \delta \gets \min_{t \le \tau \le T}\{e_{\tau}^{(k)} - \ubar{e}_{\tau}\}$; \\
                $\utilde \tau \gets \min_{1 \le \tau \le t-1} \{\tau : e_{\tau}^{(k)} - \bar{e}_{\tau} = \tilde \delta\}$; $\tilde \tau \gets \max_{t \le \tau \le T} \{\tau : e_{\tau}^{(k)} - \ubar{e}_{\tau} = \utilde \delta\}$; \\
                $\mathcal{A}^{(k+1)} \gets \mathcal{A}^{(k)} \setminus \{\hat\alpha_{s,i}\}_{s \in [\utilde \tau + 1, \tilde \tau],\, i \in [4]}$ \label{line:d2remove3}
            \ENDIF
        \ENDIF
    \ENDIF
    \STATE $k\gets k+1$
\ENDWHILE
\STATE \textbf{return} $\Delta e^{(k)}$
\end{algorithmic}
\end{algorithm}

The next proposition shows that Algorithm \ref{alg:local_search_algorithm} always returns a feasible solution:
\begin{proposition} \label{thm:GDfeasibility}
    Let $\{\Delta e^{(k)}\}$ be the sequence generated by Algorithm \ref{alg:local_search_algorithm}, and define the induced SOC $e_t^{(k)} \coloneqq e_0 + \sum_{s=1}^t \Delta e_s^{(k)}, t \in \mathcal{T}$
    and the maximum upper bound violation before time step $t$ $M_t^{(k)} \coloneqq \max_{1\le r < t} [e_r^{(k)} - \bar e_r]_+, t = \mathcal{T}\setminus \{1\}$.
    Then, for every iteration $k$, we have
    \begin{equation} \label{eq:SOCineq}
        e_t^{(k)} \ge \ubar e_t, \qquad t \in \mathcal{T},
    \end{equation}
    and
    \begin{equation} \label{eq:ubvineq}
        M_t^{(k)} \le e_t^{(k)} - \ubar e_t, \qquad t \in \mathcal{T}\setminus \{1\}.
    \end{equation}
    Consequently, every intermediate SOC solution is lower bounded by the SOC lower bound $\ubar e
    $, and every upper bound violation can still be fixed by SOC decrement. In particular, Algorithm \ref{alg:local_search_algorithm} returns a point $\Delta e^* \in \mathcal{Y}_e$.
\end{proposition}

The proof of Proposition \ref{thm:GDfeasibility} can be found in \ref{app:C}.

The next proposition shows that Algorithm \ref{alg:local_search_algorithm} indeed returns an optimal solution:
\begin{proposition} \label{thm:GDoptimality}
    The solution $\Delta e^*$ returned by Algorithm \ref{alg:local_search_algorithm} is an optimal solution of problem \eqref{eq:2ndcompact}.
\end{proposition}
The proof of Proposition \ref{thm:GDoptimality} can be found in \ref{app:D}.

Beyond optimal value, Algorithm \ref{alg:local_search_algorithm} also identifies the terminal active affine pieces of the functions $f_{t,\omega}$, which is the only second-stage information needed later for the scenario-wise subgradient computation in Section \ref{sect:subgradient}.

%% file: PSM.tex
\subsection{Projected Subgradient Method for the Two-stage Stochastic LP Problem}
\label{Projected Subgradient Method}

The greedy algorithm developed in Section~\ref{Greedy Algorithm} provides an efficient exact oracle for evaluating the second-stage LP problem in \eqref{eq:2S-LP} and computing the corresponding subgradient with respect to the first-stage decisions $p^{\DA}$. Building upon this oracle, we adopt a practical projected subgradient method to compute solutions of the two-stage stochastic LP problem \eqref{eq:2S-LP}. Since the first-stage objective is linear in $p^{\mathrm{DA}}$ and the objective of the decoupled second-stage problem is piecewise linear, the overall objective in \eqref{eq:2S-LP} is convex over the polyhedral set $\mathcal{X}$. The problem is nevertheless nonsmooth because of the imbalance settlement mechanism. Accordingly, we employ a first-order scheme that only requires scenario-wise subgradients of the second-stage value function with respect to $p^{\mathrm{DA}}$ at each iteration.

\subsubsection{Subgradient Computation via the Greedy Oracle} \label{sect:subgradient}

% For a given iterate $p^{\mathrm{DA},(k)}\in\mathcal{X}$, the greedy oracle solves the second-stage LP problem in \eqref{eq:2S-LP} for each scenario $\omega\in\Omega$ and returns both the optimal value and the corresponding active piecewise linear segments of $f_{t,\omega}$. Based on these active segments, a subgradient of \eqref{eq:2S-LP} with respect to $p^{\mathrm{DA}}$ can be obtained as follows.

For a given iterate $p^{\mathrm{DA},(k)}\in\mathcal{X}$ and scenario $\omega \in \Omega$, the greedy oracle solves the second-stage LP problem in \eqref{eq:2S-LP} for each scenario $\omega\in\Omega$ and returns an optimal second-stage solution together with the active piecewise linear segments of $f_{t,\omega}$. Since the feasible region of the second-stage LP $\mathcal{Y}_e$ is independent of $p^{\mathrm{DA}}$, and the greedy oracle returns an exact optimal solution for each scenario, a standard result on convex value functions guarantees that a subgradient of the objective evaluated at the optimal second-stage solution is a valid subgradient of the value function. Consequently, the scenario-wise subgradients computed below yield a valid subgradient of the overall two-stage objective.

We first compute a scenario-wise subgradient of the second-stage objective in \eqref{eq:2S-LP} with respect to the first-stage decision variable $p^{\mathrm{DA}}_{t,s_t(\omega)}$. Since the piecewise linear function $f_{t,\omega}(\Delta e_t, p_t^{\PV,*}(\omega), p_t^\DA)$ in the second-stage LP problem is explicitly structured, its partial derivative with respect to $p^{\mathrm{DA}}_{t,s_t(\omega)}$ can be written as
\begin{equation}
\label{eq:subgrad-slope}
\frac{\partial f_{t,\omega}}{\partial p^{\mathrm{DA}}_{t,s_t(\omega)}} = 
\begin{cases}
\lambda^{\mathrm{DA}}_{t,\omega}+\kappa, & \text{if } p_t^{\mathrm{mis}} < 0 \;\text{(net-shortfall)},\\[4pt]
\lambda^{\mathrm{DA}}_{t,\omega}-\kappa, & \text{if } p_t^{\mathrm{mis}} > 0 \;\text{(net-surplus)}.
\end{cases}
\end{equation}
Note that $f_{t,\omega}(\cdot)$ is not differentiated at $p_t^{\mathrm{mis}} = 0$, and we select the midpoint $\lambda^{\mathrm{DA}}_{t,\omega}$ as a subgradient.

Combining the first-stage objective derivative, $-\lambda^{\mathrm{DA}}_{t,\omega}$, with the selected scenario-wise subgradient of the second-stage LP objective in \eqref{eq:subgrad-slope}, the subgradient of the two-stage stochastic LP problem \eqref{eq:2S-LP} with respect to $p^{\mathrm{DA},(k)}$ is given componentwise by
\begin{equation}
\label{eq:subgrad-full}
g_{t,s}^{(k)} \;=\; \sum_{\omega\in\Omega:\, s_t(\omega)=s} \rho_\omega \,\biggl\{
-\lambda^{\mathrm{DA}}_{t,\omega}
\;+\; \frac{\partial f_{t,\omega}}{\partial p^{\mathrm{DA}}_{t,s}}\,
\biggr\},
\qquad \forall\, t\in\mathcal{T},\; s\in\mathcal{S},
\end{equation}
where the summation is taken over all scenarios $\omega$ whose Markov state at hour $t$ is $s$, thereby implicitly enforcing the nonanticipativity constraints of the first-stage offering decisions.

\subsubsection{Projection onto the Feasible Set $\mathcal{X}$}

Let $\alpha_k>0$ denote the step size at iteration $k$, whose selection rule will be specified in the next subsection. Given the current iterate $p^{\mathrm{DA},(k)}$ and the subgradient $g^{(k)}$, the subgradient update is $\tilde{p}^{\mathrm{DA}} \;:=\; p^{\mathrm{DA},(k)} - \alpha_k\,g^{(k)}$. Since the subgradient update does not necessarily preserve feasibility, $\tilde{p}^{\mathrm{DA}}$ may lie outside the feasible set $\mathcal{X}$. We therefore project it back onto $\mathcal{X}$ in the Euclidean sense, i.e., we seek the feasible point in $\mathcal{X}$ closest to $\tilde{p}^{\mathrm{DA}}$. Formally, the projection operator onto $\mathcal{X}$ is defined as
\begin{equation}
\label{eq:proj-def}
\Pi_{\mathcal{X}}(\tilde{p}^{\mathrm{DA}})
\;:=\; \argmin_{p^{\mathrm{DA}}\in\mathcal{X}} \;\frac{1}{2}\,\|p^{\mathrm{DA}}-\tilde{p}^{\mathrm{DA}}\|_2^2.
\end{equation}

Recall from \eqref{first_stage_FR} that $\mathcal{X}$ consists of (i) box constraints and (ii) monotonicity constraints. The monotonicity constraint couples the offering quantities at different price states within the same time slot, but introduces no coupling across different time slots. Since the objective in \eqref{eq:proj-def} is also separable over $t$, the projection \eqref{eq:proj-def} decomposes into $|\mathcal{T}|$ independent subproblems, one per time slot $t\in\mathcal{T}$. For a single time slot $t$, let $\{s_1,s_2,\ldots,s_N\}$ be the price states ordered so that $\pi^{\mathrm{DA}}_{t,s_1}\le \pi^{\mathrm{DA}}_{t,s_2}\le \cdots \le \pi^{\mathrm{DA}}_{t,s_N}$. Then the projection reduces to a bounded isotonic regression problem: 
\begin{subequations}
\label{eq:isotonic-proj}
\begin{align}
\min \quad &
\sum_{n=1}^{N} \frac{1}{2}\,\bigl(p_{t,s_n}^{\mathrm{DA}}-\tilde{p}_{t,s_n}^{\mathrm{DA}}\bigr)^2 \\
\text{s.t.}\quad
&\ubar{p}^{\mathrm{DA}}\le p_{t,s_n}^{\mathrm{DA}}\le \bar{p}^{\mathrm{DA}},\;\; n=1,\ldots,N. \\
& p_{t,s_1}^{\mathrm{DA}}\,\le\, p_{t,s_2}^{\mathrm{DA}}\,\le\, \cdots \,\le\, p_{t,s_N}^{\mathrm{DA}}
\end{align}
\end{subequations}

Problem \eqref{eq:isotonic-proj} is a bounded isotonic regression problem. It can be solved exactly by a standard bounded-PAVA routine, or equivalently by a generalized PAVA that enforces both monotonicity and box constraints \cite{JSSv032i05}. Intuitively, the method pools adjacent violations as in classical isotonic regression, while each pooled value is projected onto the admissible interval $[\ubar{p}^{\mathrm{DA}},\bar{p}^{\mathrm{DA}}]$. Since the subproblems are independent across time slots, the projection step is naturally parallelizable over $t\in\mathcal{T}$.

% Problem \eqref{eq:isotonic-proj} is a bounded isotonic regression problem and can therefore be solved efficiently by the pool adjacent violators algorithm (PAVA) \cite{JSSv032i05}. The basic idea is to scan the ordered components from $n=1$ to $N$: whenever a pair of adjacent components violates the monotonicity condition, the algorithm merges them into a single ``pool'' and replaces both values by their average, clipped to the box $[\ubar{p}^{\mathrm{DA}},\bar{p}^{\mathrm{DA}}]$. This pooling step is repeated until all monotonicity violations are removed. Since the subproblems are independent across time slots, the projection step is also naturally parallelizable over $t\in\mathcal{T}$.

% As a result, \eqref{eq:isotonic-proj} can be solved exactly in $\mathcal{O}(N)$ time for each time slot, yielding an overall complexity of $\mathcal{O}(N|\mathcal{T}|)$ for the full projection problem. 

\subsubsection{Step Size Selection and Termination}
\label{subsubsection:step-size}

Given the current iterate $p^{\mathrm{DA},(k)}$ and the subgradient $g^{(k)}$ in \eqref{eq:subgrad-full}, the projected subgradient update is $p^{\mathrm{DA},(k+1)}=\Pi_{\mathcal{X}} (p^{\mathrm{DA},(k)}-\alpha_k\,g^{(k)})$. Given each outer iteration already requires the solution of the exact greedy oracle, we seek a step-size rule that adds negligible overhead while adapting to the local scaling of the objective. The Barzilai-Borwein two-point formula provides a lightweight curvature proxy \cite{barzilai_twopoint_1988}. For $k\ge 1$, define $\delta p^{(k-1)} := p^{\mathrm{DA},(k)} - p^{\mathrm{DA},(k-1)},
\delta g^{(k-1)} := g^{(k)} - g^{(k-1)}$. If $\|\delta g^{(k-1)}\|_2>0$, we compute
\begin{equation}
\label{eq:spectral-step}
\alpha_k^{\mathrm{BB}} \;=\; \frac{\langle \delta p^{(k-1)},\, \delta g^{(k-1)} \rangle}
{\langle \delta g^{(k-1)},\, \delta g^{(k-1)} \rangle},
\end{equation}
and set $\alpha_k
=
\min\bigl\{\alpha_{\max},\, \max\{\alpha_{\min},\, \alpha_k^{\mathrm{BB}}\}\bigr\}$. If $\|\delta g^{(k-1)}\|_2=0$, we set $\alpha_k=\alpha_0$, where $\alpha_0\in[\alpha_{\min},\alpha_{\max}]$ is a prescribed initial step size. At the first iteration, we also use $\alpha_0$. Since the objective function $\Phi$ is convex but nonsmooth, we do not impose the monotone backtracking condition in \cite{barzilai_twopoint_1988}. Instead, the clipped BB formula above is used only to generate a safeguarded steplength for the projected subgradient iteration.

The algorithm terminates when the relative change in the objective value satisfies
\begin{equation}
\label{eq:stopping}
\frac{\bigl|\Phi\bigl(p^{\mathrm{DA},(k+1)}\bigr)-\Phi\bigl(p^{\mathrm{DA},(k)}\bigr)\bigr|}
{\max\bigl\{1,\bigl|\Phi \bigl(p^{\mathrm{DA},(k)}\bigr)\bigr|\bigr\}}
\le \varepsilon_{\mathrm{rel}}.
\end{equation}